%% file: Detlimit_Revised.tex
\RequirePackage{fix-cm}
\documentclass[smallextended,envcountsame]{svjour3}       
\smartqed  
\usepackage{graphicx}
\usepackage[reqno]{amsmath} 
\usepackage{amssymb,bbm}
\numberwithin{equation}{section}
\usepackage{tikz}
\usepackage{tkz-euclide}
\usepackage[authoryear]{natbib}
\usepackage{float}
\usepackage{pdflscape}
\usepackage{cancel}
\usepackage{pgfplots}
\usepackage{wrapfig}
\usepackage{wasysym} 
\usepackage{wrapfig}
\usepackage{comment}
\usepackage{stmaryrd}
\usepackage{caption}
\usepackage{stackengine}
\usepackage{subcaption}
\usepackage{tikz-cd}
\usepackage{nicefrac}
\usepackage{enumitem}
\usepackage{hyperref}
\usepackage{xcolor}
\usetkzobj{all}
\usetikzlibrary{shapes,arrows,spy,positioning,decorations}
\usetikzlibrary{arrows.meta,intersections}
\usetikzlibrary{decorations.pathreplacing}

\newcommand{\Gb}{\mathcal{G}}

\newcommand{\R}{{\mathbbm R}}
\newcommand{\N}{{\mathbbm N}}

\newcommand{\1}{{\mathbbm 1}}

\DeclareMathOperator{\Geom}{Geom}
\colorlet{owngrey}{gray!55}

\spnewtheorem{thm}{Theorem}{\bfseries}{\itshape}
\spnewtheorem{lem}[thm]{Lemma}{\bfseries}{\itshape}
\spnewtheorem{prop}[thm]{Proposition}{\bfseries}{\itshape}
\spnewtheorem{coro}[thm]{Corollary}{\bfseries}{\itshape}
\spnewtheorem{proposition}[thm]{Proposition}{\bfseries}{\itshape}
\spnewtheorem{definition}[thm]{Definition}{\bfseries}{\upshape}
\spnewtheorem{remark}[thm]{Remark}{\itshape}{\upshape}

\spnewtheorem*{xproof}{}{\itshape}{\rmfamily}

\renewenvironment{proof}[1][\proofname]
{\renewcommand\xproofname{#1}\xproof}
{\endxproof}

\pgfplotsset{every axis/.append style={
		axis x line=middle,    
		axis y line=middle,    
		axis line style={->}, 
	},
	cmhplot/.style={color=red,mark=none,line width=1pt,<->},
	soldot/.style={color=red,only marks,mark=*},
	holdot/.style={color=red,fill=white,only marks,mark=*},
}

\tikzset{>=stealth}

\begin{document}

\title{A probabilistic view on the deterministic mutation-selection equation: dynamics, equilibria, and ancestry via individual lines of descent
}

\author{Ellen Baake         \and
        Fernando Cordero \and
        Sebastian Hummel
}

\institute{Faculty of Technology, Bielefeld University, Box 100131, 33501 Bielefeld, Germany \\
 \and Ellen Baake \\
\email{ebaake@techfak.uni-bielefeld.de}  \\         
           \and
           Fernando Cordero \\
              \email{fcordero@techfak.uni-bielefeld.de}  \\ \and
           Sebastian Hummel \\
              \email{shummel@techfak.uni-bielefeld.de}  
}

\date{Received: date / Accepted: date}

\maketitle

\begin{abstract}
We reconsider the deterministic haploid mutation-selection equation with two types. This is an ordinary differential equation that describes the type distribution (forward in time) in a population of infinite size. This paper establishes ancestral (random) structures inherent in this deterministic model. In a first step, we obtain a representation of the deterministic equation's solution (and, in particular, of its equilibrium) in terms of an ancestral process called the killed ancestral selection graph. This representation allows one to understand the bifurcations related to the error threshold phenomenon from a genealogical point of view. Next, we characterise the ancestral type distribution by means of the pruned lookdown ancestral selection graph and study its properties at equilibrium. We also provide an alternative characterisation in terms of a piecewise-deterministic Markov process. Throughout, emphasis is on the underlying dualities as well as on explicit results.
\end{abstract}

\section{Introduction}
\label{intro}
Understanding the interplay between mutation and selection is a major topic of population genetics research. By and large, the field is divided into two major lines of research, devoted to deterministic and stochastic models, respectively. The deterministic mutation-selection equation describes the action of mutation and selection on the genetic composition of an effectively infinite population; its first version goes back to \citet{crow1956}. Deterministic mutation-selection equations are formulated in terms of discrete- or continuous-time dynamical systems, and they are treated forward in time throughout, via the well-developed methods of dynamical systems; a comprehensive overview of the research until 2000 is provided in the monograph by \citet{burger2000mathematical}. Stochastic mutation-selection models, such as the Moran and Wright-Fisher models with mutation and selection, additionally capture the fluctuations due to random reproduction over long time scales; these fluctuations are absent in the deterministic dynamics. The stochastic models have their roots in the seminal work of \citet{fisher1930genetical}, \citet{wright1931evolution}, \citet{malecotmathematiques}, \citet{feller1951diffusion}, and \citet{moran1958random}. They are formulated in terms of stochastic processes in discrete or continuous time, and are often made tractable via a diffusion limit. Their modern treatment further relies crucially on the genealogical point of view, where lines of descent are traced backward in time with the help of ancestral processes, such as the ancestral selection graph \citep{krone1997ancestral}. Overviews of the area may be found in the monographs by \citet{ewens2004mathematical}, \citet{durrett2008probability}, and \citet{wake2009}. 

During the last decades, deterministic and stochastic population genetics have largely led separate lives. It is the purpose of this article to bring the two research areas closer together by working out the backward point of view, so far reserved to stochastic models of population genetics, for deterministic mutation-selection equations. The first step in this direction was taken by \citet{Cordero2017590}; this will be our starting point. We will work with the simplest model, namely, with haploid individuals, two types, selection, and mutation, and pursue two major aims. First, we will obtain a representation of the solution of the deterministic mutation-selection equation and its equilibrium state(s) in terms of an ancestral process termed the killed ancestral selection graph. Second, we will characterise the type distribution of the ancestors of today's individuals in the distant past by what we call the pruned lookdown ancestral selection graph. Throughout, emphasis will be on the underlying dualities as well as on explicit results and worked details. 

The paper is organised as follows. We set out in Section~\ref{sec:moran} by introducing the Moran model with two types, selection, and mutation. This is a stochastic model for a finite population, which leads to the deterministic mutation-selection equation via a law of large numbers. Next, the graphical constructions required to trace back ancestral lines are introduced; namely, the ancestral selection graph (Section~\ref{s3}), the killed ancestral selection graph (Section~\ref{s4}), and the pruned lookdown ancestral selection graph (Section~\ref{s5}), all in the deterministic limit. In the special case of unidirectional mutation (away from the beneficial type, without back mutation), the results shed new light on the bifurcations related to the so-called error threshold phenomenon. Finally (Section~\ref{s6}), we characterise the ancestral type distribution in two ways: first, by means of the pruned lookdown ancestral selection graph; and second, as the absorption probability of a piecewise-deterministic Markov process.

\section{The two-type Moran model and its deterministic limit} 
\label{sec:moran}
We consider the two-type Moran model with mutation and selection, which is described as follows. We have a haploid population of fixed size~$N$. Each individual in this population has a type, which is either~$0$ or~$1$. Individuals of type~$1$ reproduce at rate~$1$, whereas individuals of type~$0$ reproduce at rate~$1+s^{}$ with~$s^{}\geq0$. We refer to type~$0$ as the fit or beneficial type, whereas type~$1$ is unfit or deleterious.
When an individual reproduces, its single offspring inherits the parent's type and replaces a uniformly-chosen individual in the population, thereby keeping the population size constant. Each individual mutates at rate~$u^{}$; the type after the event is~$i$ with probability~$\nu_{i},\ i\in \{0,1\}$. We assume throughout that~$u^{}$ is positive and~$\nu_0,\ \nu_1$ are non-negative with~$\nu_0+\nu_1=1$. 

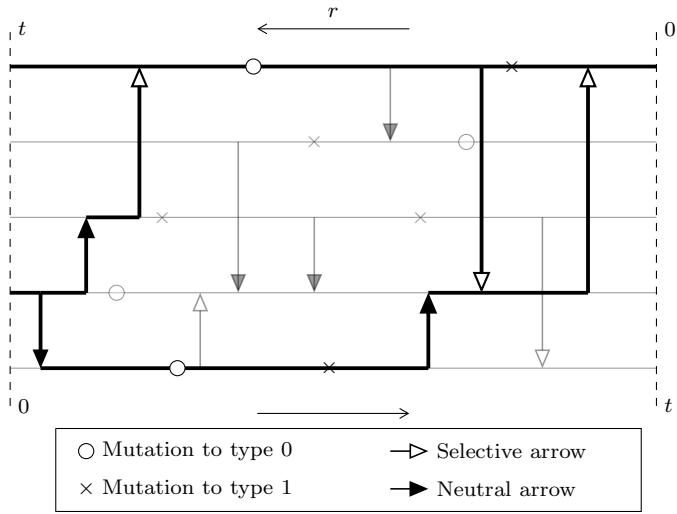
\begin{figure}[b!]
	\begin{center}
	\scalebox{1}{\begin{tikzpicture}
		\draw[line width=.5mm ] (8.5,4) -- (0,4);
		\draw[line width=.5mm ] (7.6,1) -- (5.5,1);
		\draw[line width=.5mm ] (1.7,2) -- (1,2);
		\draw[line width=.5mm ] (1,1) -- (0.4,1);
		\draw[line width=.5mm ] (5.5,0) -- (.4,0);
		\draw[line width=.5mm ] (.4,1) -- (0,1);
		\draw[line width=.5mm ] (7.6,1) -- (7.6,3.75);
		\draw[line width=.5mm ] (1.7,2) -- (1.7,3.75);
		\draw[line width=.5mm ] (1,1.8) -- (1,1);
		\draw[line width=.5mm ] (5.5,0.8) -- (5.5,0);
		\draw[line width=.5mm ] (6.2,4) -- (6.2,1.25);
		\draw[line width=.5mm ] (0.4,1) -- (.4,.2);
		
		\node[opacity=0.4] at (4,3) {$\times$} ;
		\node[opacity=0.4] at (2,2) {$\times$} ;
		\node at (6.6,4) {$\times$} ;
		\node[opacity=0.4] at (5.4,2) {$\times$} ;
		\node at (4.2,0) {$\times$} ;        
		\draw[dashed] (0,-0.5) --(0,4.5);
		\draw[dashed] (8.5,-0.5) --(8.5,4.5);
		\node [right] at (0,-0.5) {$0$};
		\node [right] at (0,4.5) {$t$};
		\node [right] at (8.5,4.5) {$0$};
		\node [right] at (8.5,-0.5) {$t$};
		\draw[-{triangle 45[scale=5]},semithick,opacity=.4] (4,2) -- (4,1) node[text=black, pos=.6, xshift=7pt]{};
		\draw[opacity=0.4] (0,1) -- (8.5,1);
		\draw[opacity=0.4] (8.5,2) -- (0,2);
		\draw[opacity=0.4] (8.5,3) -- (0,3);
		
		\draw[opacity=0.4] (0,4) -- (8.5,4);
		
		\draw[-{triangle 45[scale=5]},semithick,opacity=0.4] (5,4) -- (5,3);
		\draw[-{triangle 45[scale=5]},semithick,opacity=1] (.4,1) -- (.4,0);

		\draw[-{open triangle 45[scale=5]},thick,opacity=1] (7.6,1) -- (7.6,4);
		\draw[-{open triangle 45[scale=5]},thick,opacity=1] (1.7,2) -- (1.7,4);
		\draw[-{open triangle 45[scale=5]},thick,opacity=1] (6.2,4) -- (6.2,1);
		
		\draw[-{triangle 45[scale=5]},thick,opacity=1] (1,1) -- (1,2);
		\draw[-{triangle 45[scale=5]},thick,opacity=1] (5.5,0) -- (5.5,1);
		\draw[-{open triangle 45[scale=5]},semithick,opacity=0.4] (2.5,0) -- (2.5,1);
		\draw[-{open triangle 45[scale=5]},semithick,opacity=0.4] (7,2) -- (7,0);
		\draw[-{triangle 45[scale=5]},semithick,opacity=0.4] (3,3) -- (3,1);
		\draw[opacity=0.4] (0,0) -- (8.5,0);
		\draw (2.2,0) circle (1mm)  [fill=white!100];    
		\draw[opacity=0.4] (6,3) circle (1mm)  [fill=white!100];
		\draw (3.2,4) circle (1mm)  [fill=white!100];
		\draw (1.4,1)[opacity=0.4] circle (1mm)  [fill=white!100];
		\draw[-{angle 60[scale=5]}] (3.25,-0.6) -- (5.25,-0.6) node[text=black, pos=.5, yshift=6pt]{};
		\draw[-{angle 60[scale=5]}] (5.25,4.5) -- (3.25,4.5) node[text=black, pos=.5, yshift=6pt]{$r$};
		\draw (1,-1.1) circle (1mm)  [fill=white!100] node at (1,-1.1) [right] {\ Mutation to type~$0$};
		\node at (1 ,-1.6) {$\times$} node at (1,-1.6)[right] {\ Mutation to type~$1$};    
		\draw[-{open triangle 45[scale=2.5]},semithick] (5,-1.1) -- (5.5,-1.1) node [right] {Selective arrow};
		\draw[-{triangle 45[scale=2.5]},semithick] (5,-1.6) -- (5.5,-1.6) node [right] {Neutral  arrow};
		\draw (.6,-0.8) -- (.6,-1.9) -- (8.3,-1.9) -- (8.3,-0.8) -- (.6,-0.8);
		
		\end{tikzpicture}    }
	\end{center}
	\caption{A realisation of the Moran interacting particle system (thin lines) and the embedded ASG (bold lines). Time runs forward in the Moran model ($\rightarrow$) and backward in the ASG ($\leftarrow$).}
	\label{fig:graphical.representation.ASG}
\end{figure}

The Moran model has a well-known graphical representation as an interacting particle system, see Fig.~\ref{fig:graphical.representation.ASG}. Here, individuals are represented by pieces of horizontal lines. Time runs from left to right in the figure. Reproduction events are depicted by arrows between the lines. If an individual places offspring via an arrow, the offspring inherits the parent's type and replaces the individual at the tip. We decompose reproduction events into neutral and selective ones. This is reflected by neutral and selective arrows in the graphical representation. Neutral arrows appear at rate~$1/N$ per ordered pair of lines; selective arrows appear at rate~$s^{}/N$ per ordered pair. Neutral arrows are used by all types; selective arrows are only used by individuals of type~$0$.
Mutation events are depicted by crosses and circles on the lines. A circle (cross) indicates a mutation to type~$0$ (type~$1$), which means that the type on the line is~$0$ (is~$1$) after the mutation. This occurs at rate~$u^{}\nu_0$ (at rate~$u^{}\nu_1$) per line. Both types of mutation events are independent of the type on the line before the event; in particular, silent events are included where the type is the same before and after the event. All arrows, crosses, and circles appear independently. Given a realisation of the particle system and an initial type configuration (that is, a type assigned to each line at~$t=0$), we can read off the types on the lines at all later times~$t>0$. The distribution of the initial types is independent of the law of the graphical elements (arrows, circles, and crosses). In particular, we will see that certain properties of the ancestral processes are determined by the graphical elements alone, irrespective of the initial types.

Let~$Y^{(N)}_t$ be the proportion of type-$1$ individuals at time~$t$ in a population of size~$N$. Clearly,~$\big(Y_t^{(N)}\big)_{t\geq 0}$, which we abbreviate by~$Y^{(N)}$, is a Markov process on~$[0,1]$. We will, in what follows, study the deterministic limit of the Moran model.
In \citet[Prop.~3.1]{2015CorderoULT}, it is shown that, if~$Y_0^{(N)} \longrightarrow y_0$ as~$N \to \infty$, then~$Y^{(N)}$ converges to the solution~$y(t;y_0)$ of the initial value problem
\begin{equation}\label{eq:dlimitdiffeq}\begin{aligned}
\frac{dy}{dt}(t)&=-sy(t)\big(1-y(t)\big)-u\nu_0y(t)+u\nu_1\big(1-y(t)\big),\qquad t\geq 0,\\
y(0)&=y_0, \qquad \text{for} \ y_0\in [0,1].
\end{aligned} \end{equation}
The convergence is uniform on compact sets of time in probability and a special case of the dynamical law of large numbers of \citet[Thm.~3.1]{kurtz1970}; see also \citet[Thm.~11.2.1]{ethier1986markovprocces}. Neither parameters nor time are rescaled. This corresponds to a strong mutation--strong selection setting. (Note that this in contrast to the usual diffusion limit, where parameters and time are rescaled with population size; this is suitable in a weak mutation--weak selection framework, see, e.g., \citet[Ch.~7.2]{durrett2008probability}.) If~${\nu_0\in (0,1)}$, the convergence carries over to~$t \to \infty$ in the sense that the stationary distribution of the Moran model converges in distribution to the point measure on~$\bar{y}$ as~$N\to\infty$ \citep{2015CorderoULT}, where~$\bar{y}$ is the (unique) stable equilibrium of the ODE. The initial value problem \eqref{eq:dlimitdiffeq} is the classical mutation-selection equation of population genetics~\citep{crow1956,crow1970introduction}. It is a Riccati differential equation with constant coefficients and hence the solution is known explicitly~\citep{2015CorderoULT}. The equilibrium points of \eqref{eq:dlimitdiffeq} are the solutions of the equation
\begin{equation}\label{equieq}
s y^2-(u+s)y+u\nu_1=0.
\end{equation}
The stable equilibrium, known, for example, via \citet[Lem.~3.1]{2015CorderoULT}, is given by
\begin{equation}\label{eq:stablepoint}
\bar{y}=\begin{cases}\frac{1}{2}\left(1+\frac{u}{s}-\sqrt{\big(1-\frac{u}{s}\big)^2 + 4\nu_0\frac{u}{s}}\ \right), & s>0,\\
\nu_1, &s=0.
\end{cases}
\end{equation}
If~$s>0$, there is an additional equilibrium of \eqref{eq:dlimitdiffeq}, which is unstable~(see \citet[Lem.~3.1]{2015CorderoULT}), namely,~$$y^{\star}=
\frac{1}{2}\left(1+\frac{u}{s}+\sqrt{\left(1-\frac{u}{s}\right)^2 + 4\nu_0\frac{u}{s}}\,\right).$$
If~$\nu_0>0$, then~$\bar{y}\in [0,1)$ and~$y^{\star}>1$; so~$\bar{y}$ is the only relevant equilibrium. Furthermore,~$\lim_{t\to \infty} y(t;y_0)=\bar{y}$ for all~$y_0\in [0,1]$. If~$\nu_0=0$, the two equilibria reduce to~$\bar{y}=\min\{u/s,1\}$ and~${\nobreak y^{\star}=\max\{1,u/s\}}$. In particular, if~$\nu_0=0$ and~$u\geq s$, then~$\bar{y}$ is still the only relevant equilibrium and again~${\lim_{t\to \infty} y(t;y_0)=\bar{y}}$ for all~$y_0\in [0,1]$. But if~$\nu_0=0$ and~$u<s$, then there are two equilibria in the unit interval. In particular, then~$\lim_{t\to \infty} y(t;y_0)=\bar{y}$ for~$y_0\in [0,1)$, while~$y(t;1)\equiv1$.

For~$s>0$, Fig.~\ref{fig:stationarydistribution} shows how~$\bar y$ and~$y^{\star}$ depend on~$u/s$ and~$\nu_0$. For increasing~$u/s$, the effect of selection decreases in the sense that~$\bar y$ increases; it converges to~$\nu_1$ for~$u/s \to \infty$, which is also the equilibrium frequency when selection is absent. 
The case~$s>0,\ \nu_0=0$ deserves special attention. If~$u/s < 1$, both~$\bar{y}=u/s$~(stable) and~$y^{\star}=1$ (unstable) are in~$[0,1]$; when~$u$ surpasses the critical value~$s$, the~$\bar{y}=1$ is the only equilibrium in~$[0,1]$, and is attracting for all~$y_0 \in [0,1]$. This phenomenon is known as the error threshold \citep{Eigen1971,Eigenetal88}; it means that selection ceases to operate for~$u \geq s$. Extending \eqref{eq:dlimitdiffeq} to~$y_0\in \R$ yields a \emph{transcritical} (or \emph{exchange of stability}) bifurcation of the equilibria at~$1$ and~$u/s$: For~$u<s$, the former is unstable and the latter is stable; and vice versa for~$u>s$. See \citet{Baake1997} for more details.
Let us only add here that the equilibrium at $u/s$ in this classical mutation-selection equation with $\nu_0=0$ and $u<s$ is used to estimate fitness landscapes from molecular data via appropriate averaging~\citep{zanini2017vivo}.
\begin{figure}[t!]
	\begin{center}
		\scalebox{.35}{
			\input{stationarydistribution-fig.tex}
		}
	\end{center}
\caption{The equilibria of \eqref{eq:dlimitdiffeq} as a function of~$u/s$ for~$s>0$. Black line:~$\bar{y}$ (stable); grey~line:~$y^{\star}$~(unstable).}
\label{fig:stationarydistribution}
\end{figure}
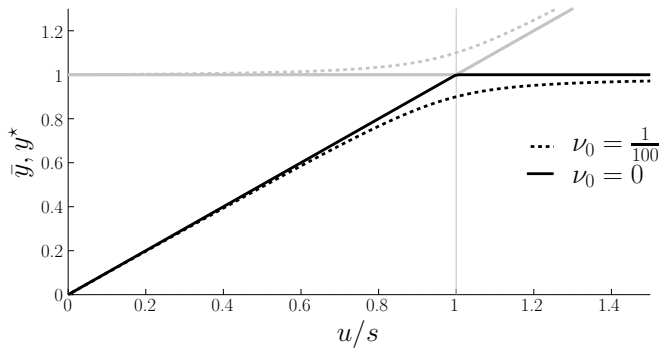

\section{The ancestral selection graph and its deterministic limit}\label{s3}

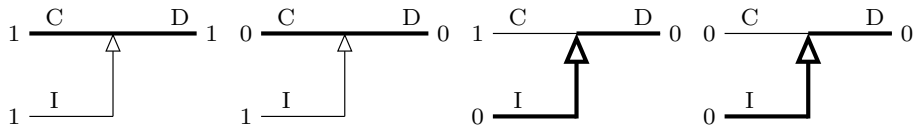
\begin{figure}[b!]
	\begin{minipage}{0.23 \textwidth}
		\centering
		\scalebox{1.1}{
			\begin{tikzpicture}
			\draw[line width=0.5mm] (0,1) -- (2,1);
			\draw[color=black] (0,0) -- (1,0);
			\draw[-{open triangle 45[scale=2.5]},color=black] (1,0) -- (1,1) node[text=black, pos=.6, xshift=7pt]{};
			\node[above] at (1.8,1) { D};
			\node[above] at (0.3,1) { C};
			\node[above] at (0.3,0) { I};
			\node[left] at (0,1) { $1$};
			\node[left] at (0,0) { $1$};
			\node[right] at (2,1) { $1$};
			\end{tikzpicture}}
	\end{minipage}\hfill
	\begin{minipage}{0.23 \textwidth}
		\centering
		\scalebox{1.1}{
			\begin{tikzpicture}
			\draw[line width=0.5mm] (0,1) -- (2,1);
			\draw[color=black] (0,0) -- (1,0);
			\draw[-{open triangle 45[scale=2.5]},color=black] (1,0) -- (1,1) node[text=black, pos=.6, xshift=7pt]{};
			\node[above] at (1.8,1) { D};
			\node[above] at (0.3,1) { C};
			\node[above] at (0.3,0) { I};
			\node[left] at (0,1) { $0$};
			\node[left] at (0,0) { $1$};
			\node[right] at (2,1) { $0$};
			\end{tikzpicture}}
	\end{minipage}\hfill
	\begin{minipage}{0.23 \textwidth}
		\centering
		\scalebox{1.1}{
			\begin{tikzpicture}
			\draw[] (0,1) -- (2,1);
			\draw[line width=0.5mm] (0,0) -- (1,0);
			\draw[line width=0.5mm] (1,1) -- (2,1);
			\draw[-{open triangle 45[scale=2.5]},color=black,line width=0.5mm] (1,-0.025) -- (1,1) node[text=black, pos=.6, xshift=7pt]{};
			\node[above] at (1.8,1) { D};
			\node[above] at (0.3,1) { C};
			\node[above] at (0.3,0) { I};
			\node[left] at (0,1) { $1$};
			\node[left] at (0,0) { $0$};
			\node[right] at (2,1) { $0$};
			\end{tikzpicture}}
	\end{minipage}\hfill
	\begin{minipage}{0.23 \textwidth}
		\centering
		\scalebox{1.1}{
			\begin{tikzpicture}
			\draw[] (0,1) -- (2,1);
			\draw[line width=0.5mm] (0,0) -- (1,0);
			\draw[line width=0.5mm] (1,1) -- (2,1);
			\draw[-{open triangle 45[scale=2.5]},color=black,line width=0.5mm] (1,-0.025) -- (1,1) node[text=black, pos=.6, xshift=7pt]{};
			\node[above] at (1.8,1) { D};
			\node[above] at (0.3,1) { C};
			\node[above] at (0.3,0) { I};
			\node[left] at (0,1) {$0$};
			\node[left] at (0,0) {$0$};
			\node[right] at (2,1) {$0$};
			\end{tikzpicture}}
	\end{minipage}
	\caption{The descendant line (D) splits into the continuing line (C) and the incoming line (I). The incoming line is ancestral if and only if it is of type~$0$. The true ancestral line is drawn in bold.}
	\label{fig:peckingorder}
\end{figure} The ancestral selection graph (ASG) by \citet{krone1997ancestral} is a tool to study the genealogical relations of a sample taken from the population at present. This is done in three steps, which we first describe for the finite-$N$ Moran model. One starts from an untyped sample (that is, no types have been assigned to the individuals) taken at~$t>0$, to which we refer as the present. In a first step, one goes backward in time and constructs a branching-coalescing graph, whose lines correspond to potential ancestors and are decorated with the mutation crosses and circles, as anticipated in Fig.~\ref{fig:graphical.representation.ASG}. When this graph has been constructed backward in time until time 0, say, one samples the types for each line without replacement from the initial type distribution. In a last step, one propagates the types forward up to time~$t$, taking into account the mutation and selection events. We will think of~$t$ as the time of sampling and denote the backward time by~$r$: Backward time~$r=0$ corresponds to the time point~$t$ and backward time~$r=t$ corresponds to the time point~$0$, as anticipated in Fig.~\ref{fig:graphical.representation.ASG}.

Let us describe these steps in detail. The branching-coalescing graph can be constructed via the graphical representation of the forward process in Fig.~\ref{fig:graphical.representation.ASG}. Choose~$n$ lines at time~$r=0$ and follow them back in time. At any time, the lines currently in the graph may be hit by selective arrows from in- or outside the current set of lines. Since we are in an untyped scenario, it is not yet possible to decide whether these arrows have been used or not. The idea is therefore to keep track of all potential ancestors of the sample. When a given line, which we call descendant line, is hit by a selective arrow, it splits into the continuing line (the one at the tip of the arrow) and the incoming line (the one at the tail). The incoming line is the ancestor if it is of type~$0$ and has thus used the selective arrow, whereas the continuing line is the ancestor if the incoming line is of type~$1$ and the selective arrow thus has not been used, see Fig.~\ref{fig:peckingorder}. 
We call this rule the \emph{pecking order}. A neutral arrow between two potential ancestors lets the two lines merge into one; this implies coalescence into a common ancestor. A coalescence event reduces the number of lines by one. 

If there are currently~$n$ lines, there is thus an increase to~$n+1$ at rate ${s(N-n) n/N}$ due to a selective arrow from one of the~$N-n$ individuals that are currently not potential ancestors. At rate~$n(n-1)/N$, there is a decrease to~$n-1$ due to a coalescence event. At rate~$sn(n-1)/N$, a selective arrow joins two potential ancestors currently in the graph. We call this a collision event. Collisions do not change the number of lines. The mutation circles and crosses occur on each line at rates~$u^{}\nu_0$ and~$u^{}\nu_1$, respectively. 

When the branching-coalescing graph has been constructed up to time~$r=t$, we sample a type for each line without replacement from the initial population with type distribution~$(1-Y^{(N)}_0,Y^{(N)}_0)$. One then propagates the types forward up to time~$t$ taking into account the pecking order and the mutations. Proceeding in this way, the types in~$[0,t]$ are determined, along with the true genealogy.

In the deterministic limit, the ASG turns into the following construction~(see \citet{Cordero2017590} for details). Branching, deleterious, and beneficial mutations occur at rate~$s, u\nu_1$, and~$u\nu_0$ per line, respectively. Since collisions and coalescences occur in the Moran model at rates of order~$\mathcal{O}(1/N)$, both types of events vanish as~$N\to \infty$. As a consequence, in the deterministic limit, all individuals in a sample remain independent in the backward process. It therefore suffices to consider a sample of size 1. For every finite time horizon, the number of lines remains bounded. The typing of the potential ancestors at~$r=t$ is done independently and identically according to the initial distribution~$(1-y_0,y_0)$.

\section{The killed ASG in the deterministic limit}\label{s4}
Our first aim now is to recover the solution of the deterministic mutation-selection equation \eqref{eq:dlimitdiffeq} by genealogical means. Recall that the solution~$y(t;y_0)$ gives the frequency of type 1 at time~$t$; the deterministic limit of the ASG is therefore the appropriate tool. Recall also that, due to the independence of the sampled individuals, it is sufficient to consider a single one. 

Our starting point is a well-known observation (e.g. \citet{Shiga1986,Athreya2005,Mano2009164}) that holds for the diffusion limit and carries over to the deterministic setting: In the absence of mutations, a single individual at time~$t$ is of type~$1$ if and only if all its potential ancestors at~$t=0$ are of type~$1$. This is easily verified via the pecking order (cf. Fig.~\ref{fig:peckingorder}). Namely, at every branching event, a type~$0$ on either the continuing or incoming line suffices for the descendant individual to be of type~$0$; iterating this over all branching events gives the statement. Mutations add further information about the types: they can determine the type of the sample even before we sample the initial types. More precisely, a mutation to type~$1$ determines the type of the line (to the right of the mutation) on which it occurs, so this line need not be traced back further into the past; it may be pruned. Next, the first mutation to type~$0$ (on any line that is still alive after the pruning) decides that the sampled individual has type~0, so that \emph{no} potential ancestor must be considered any further and the process may be killed. This motivates the following definition.

\begin{definition}
	The \textit{killed ASG in the deterministic limit} starts with one line emerging from each of the~$n$ individuals in the sample. Every line branches at rate~$s$ (due to a selective arrow from outside the set of potential ancestors). Every line is pruned at rate~$u\nu_1$ (due to a deleterious mutation). At rate~$u\nu_0$ per line, the process is killed (due to a beneficial mutation), that is, it is sent to the cemetery state~$\Delta$. All the events occur independently on every line. 
\end{definition}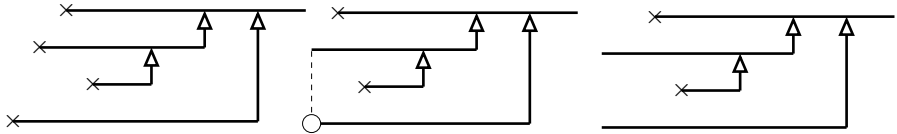
\begin{figure}[t]
\begin{minipage}{0.32\textwidth} 
	\begin{center}
		\scalebox{.7}{\begin{tikzpicture}
			\draw[line width=.5mm ] (3.5,2.1) -- (8,2.1);
			\draw[-{open triangle 45[scale=2.5]},line width=.5mm] (5.1,.7) -- (5.1,1.4);
			\draw[line width=.5mm ] (3,1.4) -- (6.1,1.4);
			\draw[-{open triangle 45[scale=2.5]},line width=.5mm] (6.1,1.4) -- (6.1,2.1);
			\draw[line width=.5mm ] (4,.7) -- (5.1,.7);
			\draw[-{open triangle 45[scale=2.5]},line width=.5mm] (7.1,0) -- (7.1,2.1);
			\draw[line width=.5mm ] (2.5,0) -- (7.1,0);
			\node at (3.5,2.1) {\scalebox{1.6}{$\times$}} ;
			\node at (3,1.4) {\scalebox{1.6}{$\times$}} ;
			\node at (4,.7) {\scalebox{1.6}{$\times$}} ;
			\node at (2.5,0) {\scalebox{1.6}{$\times$}} ;
			\draw[dashed,opacity=0] (2,0) --(2,2.1);
			\end{tikzpicture}}
		
	\end{center}
\end{minipage}
\hfill 
\begin{minipage}{0.3\textwidth} 
	\begin{center}
		\scalebox{.7}{\begin{tikzpicture}
			\draw[line width=.5mm ] (3.5,2.1) -- (8,2.1);
			\draw[-{open triangle 45[scale=2.5]},line width=.5mm] (5.1,.7) -- (5.1,1.4);
			\draw[line width=.5mm ] (3,1.4) -- (6.1,1.4);
			\draw[-{open triangle 45[scale=2.5]},line width=.5mm] (6.1,1.4) -- (6.1,2.1);
			\draw[line width=.5mm ] (4,.7) -- (5.1,.7);
			\draw[-{open triangle 45[scale=2.5]},line width=.5mm] (7.1,0) -- (7.1,2.1);
			\draw[line width=.5mm ] (3,0) -- (7.1,0);
			
			\node at (3.5,2.1) {\scalebox{1.6}{$\times$}} ;
			\node at (4,.7) {\scalebox{1.6}{$\times$}} ;
			\draw[dashed] (3,0) --(3,1.4);
			\draw (3,0) circle (1.7mm)  [fill=white!100];
			\end{tikzpicture}}
	\end{center}
\end{minipage}
\hfill
\begin{minipage}{0.32\textwidth} 
	\begin{center}
		\scalebox{.7}{\begin{tikzpicture}
			\draw[line width=.5mm ] (3.5,2.1) -- (8,2.1);
			\draw[-{open triangle 45[scale=2.5]},line width=.5mm] (5.1,.7) -- (5.1,1.4);
			\draw[line width=.5mm ] (2.5,1.4) -- (6.1,1.4);
			\draw[-{open triangle 45[scale=2.5]},line width=.5mm] (6.1,1.4) -- (6.1,2.1);
			\draw[line width=.5mm ] (4,.7) -- (5.1,.7);
			\draw[-{open triangle 45[scale=2.5]},line width=.5mm] (7.1,0) -- (7.1,2.1);
			\draw[line width=.5mm ] (2.5,0) -- (7.1,0);
			
			\node at (3.5,2.1) {\scalebox{1.6}{$\times$}} ;
			\node at (4,.7) {\scalebox{1.6}{$\times$}} ;
			\end{tikzpicture}}
	\end{center}
\end{minipage}
\caption{The killed ASG either absorbs in a state with~$0$ lines due to mutations to type~$1$ (left) or in a cemetery state~$\Delta$ due to a mutation to type~$0$ (center); it may also grow to~$\infty$ (not shown). The realisation on the right is still in a transient state.}
\label{fig:Rstates}
\end{figure}
Fig.~\ref{fig:Rstates} depicts some realisations of the killed ASG. There, we adopt the convention that the incoming line is always placed immediately beneath the continuing line. Let~$(R_r)_{r\geq 0}$ be the line-counting process of the killed ASG. This is a continuous-time Markov chain with values in~$\N_{0}^{\Delta}:=\N_0\cup \{\Delta\}$ and transition rates 
\begin{equation}\label{eq:Rrates}
q^{}_R(k,k+1) = ks, \qquad q^{}_R(k,k-1) = k u \nu^{}_1, \qquad q^{}_R(k,\Delta) = ku \nu^{}_0
\end{equation}
for~$k \in \N_0$.
The states~$0$ and~$\Delta$ are absorbing; all other states are transient. The state~$0$ is reached if all lines are pruned due to deleterious mutations. The state~$\Delta$ is reached upon the first beneficial mutation. Absorption in~$0$ (in~$\Delta$) implies that (not) all individuals in the sample are of type~$1$. The process may also grow to~$\infty$ (this happens with positive probability if~$\nu_0=0, u<s$).

We now establish a connection between the solution~$y(\cdot\,;y_0)$ of the (deterministic) mutation-selection equation and the (stochastic) line-counting process~$(R_r)_{r\geq 0}$, in terms of a duality relation, which formalises the ideas described above. Let~$H:[0,1]\times \N_{0}^{\Delta}\to\R$ be defined as \begin{equation}\label{eq:dualityfct}
H(y_0,n)=y_0^n, \quad \text{for }y_0\in [0,1],\ n\in \N_0^{\Delta},
\end{equation}
where~$y_0^{\Delta}:=0$ for all~$y_0\in[0,1]$. The function~$H$ returns the sampling probability for~$n$ individuals to be of type~$1$ under~$y_0$. Setting~$y_0^{\Delta}=0$ is in accordance with this interpretation: it is impossible to sample an unfit individual that has a beneficial mutation in its relevant ancestry. The function~$H$ will serve as our duality function.
\begin{thm}\label{thm:dualityRY}
	The line-counting process~$(R_r)_{r\geq 0}$ of the killed ASG and the solution~$y(\cdot\,;y_0)$ of the deterministic mutation-selection equation \eqref{eq:dlimitdiffeq} satisfy the duality relation \begin{equation}\label{eq:dualityRY}
	y(t;y_0)^n=E[y_0^{R_t} \mid R_0=n]\quad \text{for all } n\in \N_0^{\Delta},\ y_0\in [0,1], \ \text{and}\ t\geq 0.
	\end{equation}
\end{thm}
\begin{proof}
	We can consider~$\big(y(t;y_0)\big)_{t\geq 0}$ as a (deterministic) Markov process on~$[0,1]$ with generator $$
	\mathcal{A}_y^{}f(y)= \mathcal{A}_y^sf(y)+\mathcal{A}_y^uf(y)
	$$
	for~$f\in C^1([0,1],\R)$, where \begin{equation}
	\mathcal{A}_y^sf(y):=-sy(1-y)\frac{\partial f}{\partial y} \quad \text{and} \quad \mathcal{A}_y^uf(y):=[-u\nu_0y+u\nu_1(1-y)]\frac{\partial f}{\partial y}\label{eq:generatorYparts}\end{equation} correspond to selection and mutation, respectively. On the other hand, the infinitesimal generator of the line-counting process of the killed ASG reads \begin{equation}
	\mathcal{A}^{}_R\tilde{f}(n)=\mathcal{A}_R^s\tilde{f}(n)+\mathcal{A}_R^u\tilde{f}(n)\label{eq:generatorRparts}
	\end{equation} for~$\tilde{f}\in C_b(\N^{\Delta}_0,\R)$, where $$\mathcal{A}_R^s\tilde{f}(n):=ns[\tilde{f}(n+1)-\tilde{f}(n)]$$ and $$\mathcal{A}_R^u\tilde{f}(n):=nu\nu_1[\tilde{f}(n-1)-\tilde{f}(n)]+nu\nu_0[\tilde{f}(\Delta)-\tilde{f}(n)]$$ again correspond to selection and mutation, respectively. 
	Since~$H$ is continuous, it suffices to show that $$\mathcal{A}^{}_yH(\cdot,n)(y)=\mathcal{A}_R^{}H(y,\cdot)(n)\qquad \text{for}\  y\in [0,1]\ \text{and}\ n\in \N_0^{\Delta}$$ to prove the duality (see \citet[Thm.~3.42]{liggett2010} or \citet[Prop.~1.2]{jansen2014}). 
	This matching of the generators is a straightforward calculation and can be done individually for the selection and mutation parts; for example, 	$$\mathcal{A}_y^sH(\cdot,n)(y)=-ns[y^n-y^{n+1}]=\mathcal{A}_R^sH(y,\cdot)(n).$$
	Similarly,~$\mathcal{A}_y^uH(\cdot,n)(y)=\mathcal{A}_R^u H(y,\cdot)(n)$.  \qed
\end{proof}
Theorem~\ref{thm:dualityRY} provides a stochastic representation of the solution of the deterministic mutation-selection equation. It tells us that the killed ASG is indeed the right process to determine the current type distribution. To see this, set~$n=1$ and note that the right-hand side of~\eqref{eq:dualityRY} indeed equals the probability that a single individual at time~$t$ is of type~$1$: This is the case if either all lines have been pruned before time~$t$; or if all lines still alive at time~$t$ are assigned type~$1$ when sampling from the initial distribution with weights~$(1-y_0,y_0$), see Fig.~\ref{fig:Rstates}. 
\begin{remark}\label{rem:strongdualityforwardpicture}
	Theorem~\ref{thm:dualityRY} amounts to a weak duality between the forward and the backward process. We expect that this also holds pathwise (see \citet[Sect.~4]{jansen2014} for the corresponding notions). But in order to establish this strong kind of duality, one would need a particle construction of the forward process (such as a lookdown construction as in~\citet{donnelly1999}, but for the deterministic limit). This is beyond the scope of this article.
\end{remark}

We are particularly interested in the equilibrium~$\bar{y}$. Let us note in passing:
\begin{coro}
	$$\big(\bar{y}^{R_r}\big)_{r\geq 0} \text{is a martingale.}$$
\end{coro}
\begin{proof} 
	Setting~$y_0=\bar{y}$ in \eqref{eq:dualityRY} yields $$E[\bar{y}^{R_r} \mid R_0=n] = \bar{y}^n,$$
	which implies that the left-hand side does not depend on~$r$. \qed
\end{proof}
We now proceed to recover~$\bar{y}$ via the probabilistic backward picture. To this end, we take the limit~$t\to\infty$ in \eqref{eq:dualityRY}. This leads us to consider the asymptotic behaviour of~$R$, which is stated in the following Lemma.
\begin{lem}\label{lem:condabsorbR}
	\begin{enumerate}[label=(\roman*),leftmargin=25pt]
		\item If~$\nu_0=1$,~$R$ absorbs in~$\Delta$ with probability~$1$.
		\item If~$\nu_0\in(0,1)$,~$R$ absorbs in~$\{0,\Delta\}$ with probability~$1$. 
		\item If~$\nu_0=0$ and~$u<s$,~$R$ absorbs in~$0$ with probability~$<1$ and, conditional on non-absorption of~$R$ in~$0$,~$R_r\to\infty$ with probability~$1$.
		\item If~$\nu_0=0$ and~$u\geq s$,~$R$ absorbs in~$0$ with probability~$1$. 
	\end{enumerate}
\end{lem}
\begin{proof}
	If~$\nu_0>0$, conditional on non-absorption of~$R$ in~$\{0,\Delta\}$, there is always at least one line in the killed ASG. The time to the first beneficial mutation on any given line is exponentially distributed with parameter~$u\nu_0$ and therefore finite almost surely. Hence, for any~$n\in \N$, $$P(R_r\notin\{0,\Delta\} \mid R_0=n)\leq P(R_r\neq \Delta\mid R_r\neq 0,R_0=n) \stackrel{r\to\infty}{\longrightarrow} 0.$$ This proves~$(ii)$. If~$\nu_0=1$, we have~$\nu_1=0$ and hence~$P(R_r=0\mid R_0=n)=0$ for all~$n\in \N$ and~$r\geq 0$. The same argument used for~$(ii)$ then leads to~$(i)$. 
	For the second statement of~$(iii)$, note that conditional on non-absorption of~$R$ in~$0$,~$R$ is transient and hence we can apply \citet[Thm.~8]{KarlinMcGregor57}. The other cases follow by the classical absorption criterion \citep[Sect.~5]{KarlinMcGregor57}.\qed
\end{proof}
Setting~$n=1$ in \eqref{eq:dualityRY} and taking the limit~$t\to\infty$, we directly obtain a representation of the equilibrium frequency~$\bar{y}$ in terms of the absorption probability of~$R$ in~$0$. 
\begin{coro}\label{coro:absorptionrbary}
	\begin{equation}\bar{y}=P(\lim_{r\to \infty} R_r=0 \mid R_0=1).\end{equation}
\end{coro}
Therefore, we can now recover \eqref{eq:stablepoint} using only properties of~$R$. To calculate the absorption probabilities, let~$w_n:=P(\lim_{r\to \infty} R_r=0 \mid R_0=n)$. A first-step decomposition yields \begin{equation} w_n=\frac{s}{u+s}w_{n+1}+\frac{u\nu_1}{u+s}w_{n-1},\qquad n\geq 1,\label{eq:quadeqrt} \end{equation} together with~$w_{0}=1$ and~$w_{\Delta}=0$. It remains to solve \eqref{eq:quadeqrt}.
Due to the independence of the~$n$ lines, one has~$w_n=w_1^n,$ and it suffices to show the following.
\begin{prop}\label{eq:c1=bary}
	\begin{equation}\label{eq:w1}
	w_1= \begin{cases}
	\frac{1}{2}\left(1+\frac{u}{s}-\sqrt{\big(1-\frac{u}{s}\big)^2 + 4\nu_0\frac{u}{s}}\ \right), & s>0,\\
	\nu_1, &s=0.
	\end{cases}
	\end{equation}
\end{prop}
\begin{remark}
	Note that, for~$\nu_0=0$, \eqref{eq:w1} reduces to $$w_1=\begin{cases}
	\min\big\{\frac{u}{s},1\big\}, &\text{if }s>0,\\
	1,&\text{if }s=0.
	\end{cases}$$
\end{remark}
\begin{proof}[Proof of Proposition~\ref{eq:c1=bary}] Using the product form of~$w_2$, \eqref{eq:quadeqrt} evaluated for~$n=1$ leads to 
	\begin{equation}\label{eq:fixed_point}
	sw_1^2-(u+s)w_1+u\nu_1=0,
	\end{equation}
	i.e.~$w_1$ satisfies Eq. \eqref{equieq}. In particular, for~$s=0$, one has~$w_1=\nu_1$. For~$s>0$, we get~$w_1\in\{\bar{y},y^\star\}$. In addition, if~$\nu_0>0$ or~$u>s$, we already know from Section~\ref{sec:moran} that~$y^\star>1$. Since~$w_1$ is a probability, we therefore have~$w_1=\bar{y}$. If~$\nu_0=0$ and~$u< s$, then~$y^\star=1$ and~$\bar{y}<1$. But Lemma~\ref{lem:condabsorbR} implies~$w_{1}<1$, so~$w_1=\bar{y}$. Finally, if~$\nu_0=0$ and~$s=u$, we have~$w_1=\bar{y}=y^\star$.\qed\end{proof}

Since~$w_1=\bar{y},$ Proposition~\ref{eq:c1=bary} is in accordance with Corollary~\ref{coro:absorptionrbary}. We have thus found the desired genealogical interpretation of the solution of the deterministic mutation-selection equation \eqref{eq:dlimitdiffeq} and, in particular, of its stable equilibrium~$\bar y$. Let us explicitly describe what happens in the special case~$\nu_0=0$, which brings about the bifurcation that corresponds to the error threshold.
In this case,~$\Delta$ cannot be accessed,~$R$ is a birth-death process with birth rate~$s$ and death rate~$u$, and~$w_1=\bar{y}$ corresponds to its extinction probability. Namely, for~$u \geq s$, the process dies out almost surely, whereas for~$u < s$, it survives with positive probability~$1 - u/s$ and then grows to infinite size almost surely. This is a classical result from the theory of branching processes \citep[Ch.~III.4]{athreyaney1972}: Indeed, for~$\nu_0=0$, \eqref{eq:fixed_point} is the fixed point equation~$w_1 = \varphi(w_1)$ for the generating function~$\varphi$ of the offspring distribution of a binary Galton-Watson process with probability~$u/(u + s)$ for no offspring and~$s/(u + s)$ for two offspring individuals. This connection sheds new light on the bifurcation observed in Section~\ref{sec:moran} and Fig.~\ref{fig:stationarydistribution}. Namely, let us consider the killed ASG starting from a single individual sampled from the equilibrium population (at some late time~$t$, say), so~$R_0 = 1$. If~$R$ converges to~$\infty$, then~$\lim_{r\to\infty} y_0^{R_r}=0$ for all~$y_0\in [0,1)$, so the sampled individual is of type 0; whereas~$y_0^{R_r}\equiv 1$ for~$y_0=1$ and~$r\geq 0$, which results in an individual of type 1. On the other hand, conditional on eventual absorption of~$R$ in~$0$,~$\lim_{r\to\infty}y_0^{R_r}=1$ for all~$y_0\in [0,1]$, which renders type 1 for the sampled individual.

\section{The pruned lookdown ASG in the deterministic limit}\label{s5}
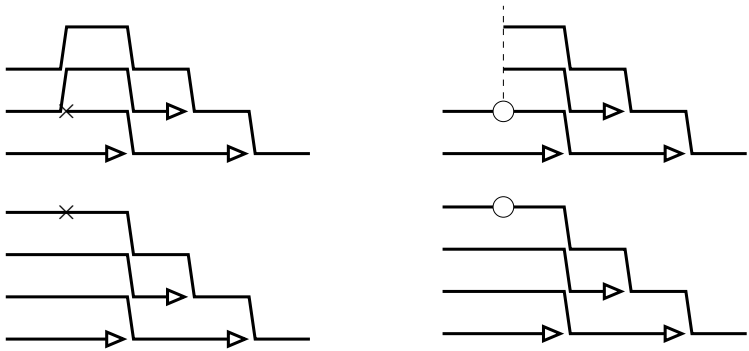
\begin{figure}[b]
	\begin{minipage}{0.5\textwidth} 
		\begin{center}
			\scalebox{.8}{\begin{tikzpicture}
				\draw[line width=.5mm ] (0,1.4) -- (.9,1.4) -- (1,2.1) -- (2,2.1) -- (2.1,1.4) -- (3,1.4) -- (3.1,.7) -- (4,.7) -- (4.1,0) -- (5,0);
				\draw[-{open triangle 45[scale=2.5]},line width=.5mm] (1,.7) -- (2,.7) -- (2.1,0) -- (4,0);
				\draw[-{open triangle 45[scale=2.5]},line width=.5mm] (0,.7) -- (0.9,.7) -- (1,1.4) -- (2,1.4) -- (2.1,.7) -- (3,.7);
				\draw[-{open triangle 45[scale=2.5]},line width=.5mm] (0,0) -- (2,0);
				\node at (1,.7) {\scalebox{1.6}{$\times$}} ;
				
				\draw[dashed,opacity=0] (1,.7) --(1,2.45);
				\end{tikzpicture}}
		\end{center}
	\end{minipage}
	\begin{minipage}{0.45\textwidth} 
		\begin{center}
			\scalebox{.8}{\begin{tikzpicture}
				\draw[line width=.5mm ] (1,2.1) -- (2,2.1) -- (2.1,1.4) -- (3,1.4) -- (3.1,.7) -- (4,.7) -- (4.1,0) -- (5,0);
				\draw[-{open triangle 45[scale=2.5]},line width=.5mm] (0,.7) -- (1,.7) -- (2,.7) -- (2.1,0) -- (4,0);
				\draw[-{open triangle 45[scale=2.5]},line width=.5mm] (1,1.4) -- (2,1.4) -- (2.1,.7) -- (3,.7);
				\draw[-{open triangle 45[scale=2.5]},line width=.5mm] (0,0) -- (2,0);
				\draw[dashed,opacity=100] (1,.7) --(1,2.45);
				\draw (1,.7) circle (1.7mm)  [fill=white!100];
				\end{tikzpicture}}
		\end{center}
	\end{minipage}
	\hfill \\ ~~\vspace{.4cm} \\
	\begin{minipage}{0.5\textwidth} 
		\begin{center}
			\scalebox{.8}{\begin{tikzpicture}
				\draw[line width=.5mm ] (0,2.1) -- (1,2.1) -- (2,2.1) -- (2.1,1.4) -- (3,1.4) -- (3.1,.7) -- (4,.7) -- (4.1,0) -- (5,0);
				\draw[-{open triangle 45[scale=2.5]},line width=.5mm] (0,.7) -- (2,.7) -- (2.1,0) -- (4,0);
				\draw[-{open triangle 45[scale=2.5]},line width=.5mm] (0,1.4) -- (1,1.4) -- (2,1.4) -- (2.1,.7) -- (3,.7);
				\draw[-{open triangle 45[scale=2.5]},line width=.5mm] (0,0) -- (2,0);
				\node at (1,2.1) {\scalebox{1.6}{$\times$}} ;
				
				\draw[dashed,opacity=0] (1,.7) --(1,2.45);
				\end{tikzpicture}}
		\end{center}
	\end{minipage}
	\begin{minipage}{0.45\textwidth} 
		\begin{center}
			\scalebox{.8}{\begin{tikzpicture}
				\draw[line width=.5mm ] (0,2.1) -- (1,2.1) -- (2,2.1) -- (2.1,1.4) -- (3,1.4) -- (3.1,.7) -- (4,.7) -- (4.1,0) -- (5,0);
				\draw[-{open triangle 45[scale=2.5]},line width=.5mm] (0,.7) -- (2,.7) -- (2.1,0) -- (4,0);
				\draw[-{open triangle 45[scale=2.5]},line width=.5mm] (0,1.4) -- (1,1.4) -- (2,1.4) -- (2.1,.7) -- (3,.7);
				\draw[-{open triangle 45[scale=2.5]},line width=.5mm] (0,0) -- (2,0);
				\draw (1,2.1) circle (1.7mm)  [fill=white!100];
				\end{tikzpicture}}
		\end{center}
	\end{minipage}
	\caption{The pruned lookdown ASG: Pruning due to a deleterious mutation on a line that is not at the top~(top left); pruning of all lines above a beneficial mutation~(top right); a deleterious and a beneficial mutation at the top line, which do not affect the number of potential ancestors (bottom).}
	\label{fig:pldasgtransitions}
\end{figure}
Let us now turn to the type of the \emph{ancestor} of a single individual from the equilibrium population. This is a more involved problem than identifying the (stationary) type distribution of the forward process, because we now must identify the parental branch (incoming or continuing, depending on the type) at every branching event, which requires nested case distinctions. Furthermore, some ancestral lines must be traced back beyond the first mutation. Nevertheless, mutations may still rule out certain potential ancestors. To describe this, \citet{Cordero2017590} extended the pruned lookdown ASG (pLD-ASG) of \citet{Lenz201527} to the framework of the deterministic limit. Let us recall the idea behind this process. The pLD-ASG starts from a single individual. The lines of the graph correspond to the potential ancestors and are assigned consecutive levels, starting at level~$1$ (see Fig.~\ref{fig:pldasgtransitions}). If a line is hit by a selective arrow, its level is increased by one and at the same time all lines above it are shifted up one level; thereby making space for the incoming line, which then occupies the former level of the line it hit. If the first event on a line that does not occupy the top level is a mutation to type~$1$, we can conclude that it will not be ancestral, since it will, at a later time, play the role of an unsuccessful incoming line, for its type is~$1$ due to the mutation. Hence we can cut away this line. The line occupying the top level is exempt from the pruning since, regardless of its type, this line will be ancestral if all lines below it are non-ancestral. If a line that is not the top line has a mutation to type~$0$, we can cut away all lines above it, because this line will, at some stage, be an incoming line and will, due to the mutation, succeed against lines above it. If the top line is hit by a mutation to type~$0$, this does not have an effect. This motivates the following definition. 
\begin{definition}\label{def:pLDASG}
	The \textit{pruned lookdown ASG in the deterministic limit} starts at time~$r=0$ and proceeds in direction of increasing~$r$. At each time~$r$, the graph consists of a finite number~$L_r$ of lines. 
	The lines are numbered by the integers~$1,\ldots, L_r$, to which we refer as \emph{levels}. The process then evolves via the following transitions.
	\begin{enumerate} 
		\item
		Every line~$i \leqslant L_r$ branches at rate~$s$ and a new line, namely the incoming branch, is inserted at level~$i$ and all lines at levels~$k\geqslant i$ 
		are pushed one level upward to~$k+1$; in particular, the continuing branch is shifted from level~$i$ 
		to~$i+1$. $L_r$ increases to~$L_{r}+1$. 
		\item 
		Every line~$i \leqslant L_r$ experiences deleterious mutations at rate~$u \nu_1$. If~$i=L_r$, nothing happens. If~$i<L_r$, the line at level~$i$ is pruned, and the lines above it slide down to `fill the gap', rendering the transition from~$L_r$ to~$L_{r}-1$. 
		\item 
		Every line~$i \leqslant L_r$ experiences beneficial mutations at rate~$u \nu_0$. All the lines at levels~$> i$ are pruned, resulting in a transition from~$L_r$ to~$i$. Thus, no pruning happens if a beneficial mutation occurs on level~$L_r$.
	\end{enumerate}
	All the events occur independently on every line. We call~$L=(L_r)_{r\geq0}$ the line-coun\-ting process of the pLD-ASG.
\end{definition}
The line-counting process~$L$ is a Markov chain on~$\N$, the transition rates of which result directly from the definition as
\begin{equation}\label{eq:pLDASGrates}
q^{}_L(n,n+1) = n s, \quad q^{}_L(n,n-1) = (n-1) u \nu^{}_1 + \1_{\{n>1\}} u \nu^{}_0, \quad q^{}_L(n,n-\ell) = u \nu^{}_0, 
\end{equation}
where~$2\leq\ell<n$. 
\begin{remark}
	For later use, we do not insist on starting from a single individual; but one should keep in mind that if we start the process with~$n>1$ lines, then it does not correctly describe the ancestry of~$n$ individuals. For example, assume that the first event is a beneficial mutation on line~$1$. This induces pruning of all other lines, which is incompatible with the ancestry of~$n$ individuals.
\end{remark}For any given~$r>0$, a hierarchy is, by construction, imposed on the lines of the graph, such that if one line is~$0$, the lowest line occupied by a type-$0$ individual is the true ancestral line. In particular, the ancestor at time~$r$ is then of type~$0$. If all lines are occupied by individuals of type~$1$, the top line is the true ancestral line and the ancestor at time~$r$ is of type~$1$. In the finite Moran model and in the diffusion limit, the line that is ancestral if all potential ancestors are of type~$1$ is called immune (the name originates from the immunity to pruning by deleterious mutations); in the deterministic limit, the immune line is always the top line. 

The above rationale may be used to determine the ancestor's type at any time~$t=r$; but explicit results require the limit~$r \to \infty$. We therefore now consider the
asymptotic behaviour of~$L_r$. Recall that we assume~$u>0$ throughout. 
\begin{prop} \label{prop:asymptoticLbeh}
	\begin{enumerate}[label=(\roman*),leftmargin=25pt]
		\item If~$s=0$,~$L$ absorbs in~$1$ almost surely.
		\item If~$u<s$ and~$\nu_0=0$,~$L$ is transient, so~$L_r \to \infty$ almost surely as~$r \to \infty$.
		\item If~$u=s$ and~$\nu_0=0$,~$L$ is null recurrent.
		\item If~$u> s$ or~$\nu_0>0$,~$L$ is positive recurrent and the stationary distribution is geometric with parameter~$1-p$, where $$p=\begin{cases}
		\frac{s}{u\nu_1}\bar{y}, &\text{if }\nu_1>0,\\
		\frac{s}{u+s},&\text{if }\nu_1=0.	\end{cases}$$
	\end{enumerate}
\end{prop}
\begin{remark}\label{rem:p}
	The parameter of the geometric distribution~$p=p(u,s,\nu_1)$ is a function of~$u,s,$ and~$\nu_1$. Explicitly, it is given by 
	\begin{equation} p(u,s,\nu_1)=\begin{cases}
	\frac{1}{2}\Big(\frac{u+s}{u\nu_1}-\sqrt{\big(\frac{u+s}{u\nu_1}\big)^2-4\frac{s}{u\nu_1}}\ \Big) , &\text{if }\nu_1>0,\\
	\frac{s}{u+s},&\text{if }\nu_1=0.	\end{cases}\end{equation}
	It is continuous in~$\nu_1$, i.e.~$\lim_{\nu_1\to 0}p(u,s,\nu_1)=p(u,s,0).$ 
\end{remark}
\begin{remark}
	A proof for case~$(iv)$, for~$\nu_1>0$, is given in \citet[Lem.~5.3]{Cordero2017590}.
\end{remark}

\begin{proof}[Proof of Proposition~\ref{prop:asymptoticLbeh}]
	Case~$(i)$ is trivial. Cases~$(ii)$ and~$(iii)$ are straightforward applications of \citet[Thm.~1,Thm.~2]{KarlinMcGregor57}.
	For case~$(iv)$, note that~$L$ is stochastically dominated by a Yule process with branching rate~$s$. This Yule process is non-explosive. 
	One easily checks that the claimed geometric distribution is invariant. Every process which is non-explosive and has an invariant distribution is positive recurrent, see \citet[Thm.~3.5.3]{norris1998markov}. The uniqueness of the stationary distribution follows from \citet[Thm.~3.5.2]{norris1998markov}.\qed
\end{proof}
In what follows, the asymptotic tail probabilities of~$L_r$ are crucial. Let $$a_n:=\lim_{r\to\infty} P_1(L_{r}>n)$$ if this limit exists (the subscript denotes the initial value). We first focus on the positive recurrent case where we know the limit exists. 
\begin{prop} \label{prop:fearnheadrec}
	If~$L$ is positive recurrent, the coefficients~$(a_n)_{n\geq 0}$ satisfy	\begin{equation}
	a^{}_n=\frac{s}{u+s}a^{}_{n-1}+\frac{u\nu_1}{u+s}a^{}_{n+1},\qquad n\in \N,\label{eq:fearnheadrecursion}\end{equation}
	with boundary condition~$a^{}_0=1$ and~$\lim_{n\to \infty}a^{}_{n}=0$.
\end{prop}
\begin{remark}
	Recursion \eqref{eq:fearnheadrecursion} is the analogue to Fearnhead's recursion \citep{fearnhead2002common} in the deterministic limit.
\end{remark}
\begin{remark}\label{rem:interchangesnu1}
	If we interchange the roles of~$s$ and~$u\nu_1$ in \eqref{eq:fearnheadrecursion} and replace the boundary condition~$\lim_{n\to\infty}a_n=0$ by~$a_{\Delta}=0$, we obtain the recursion for~$w^{}_n$ in \eqref{eq:quadeqrt} (note that~$u=u\nu_0+u\nu_1$ such that~$u+s$ is invariant under the interchange of~$s$ and~$u\nu_1$).
\end{remark}
\begin{figure}[t]
	\begin{minipage}{0.32\textwidth} 
		\begin{center}
			\scalebox{.8}{\begin{tikzpicture}
				\draw[line width=.5mm ] (0,.7) -- (2,.7) -- (2.1,0) -- (3,0);
				\draw[line width=.5mm ] (0,2) -- (2,2) -- (2.1,1.3) -- (3,1.3);
				
				\draw (1,1) circle (.7mm)   [fill=black!100];
				\draw (1,1.3) circle (.7mm)  [fill=black!100];
				\draw (1,1.6) circle (.7mm)    [fill=black!100];
				\draw[-{open triangle 45[scale=2.5]},line width=.5mm] (0,0) -- (2,0);
				\node[left] at (0,2) {$n+1$};
				\node[right] at (3,1.3) {$n$};
				\draw[dashed,opacity=0] (1,0) --(1,2.7);
				\draw[dotted,opacity=1] (2,0) --(2,-.5);
				\node[left] at (2,-0.5) {$T$};
				\draw[dotted,opacity=1] (2.2,0) --(2.2,-.5);
				\node[right] at (2.2,-0.5) {$T-$};
				\end{tikzpicture}}
		\end{center}
	\end{minipage}
	\hfill
	\begin{minipage}{0.32\textwidth} 
		\begin{center}
			\scalebox{.8}{\begin{tikzpicture}
				\draw[line width=.5mm ] (0,0) -- (2,0) -- (2.1,.7) -- (3,.7);
				\draw[line width=.5mm ] (0,1.3) -- (2,1.3) -- (2.1,2) -- (3,2);
				\draw[line width=.5mm ] (0,2) -- (2,2) -- (2.1,2.7) -- (3,2.7);
				\draw[line width=.5mm ] (2.1,0) -- (3,0);
				
				\draw (1,.3) circle (.7mm)  [fill=black!100];
				\draw (1,.6) circle (.7mm)  [fill=black!100];
				\draw (1,.9) circle (.7mm)  [fill=black!100];
				\node at (2.1,0) {\scalebox{1.6}{$\times$}} ;
				\node[left] at (0,2) {$n+1$};
				\node[right] at (3,2.7) {$n+2$};
				\draw[dashed,opacity=0] (1,0) --(1,2.7);
				\draw[dotted,opacity=1] (2.1,0) --(2.1,-.5);
				\node[left] at (2.1,-0.5) {$T$};
				\draw[dotted,opacity=1] (2.3,0) --(2.3,-.5);
				\node[right] at (2.3,-0.5) {$T-$};
				\end{tikzpicture}}
		\end{center}
	\end{minipage}\hfill 
	\begin{minipage}{0.32\textwidth} 
		\begin{center}
			\scalebox{.8}{\begin{tikzpicture}
				\draw[line width=.5mm ] (0,0) -- (3,0);
				\draw[line width=.5mm ] (0,.7) -- (3,.7);
				\draw[line width=.5mm ] (2,2) -- (3,2);
				\draw[line width=.5mm, opacity=0.2] (0,2) -- (2,2);
				
				\draw (1,1) circle (.7mm)  [fill=black!100];
				\draw (1,1.3) circle (.7mm)  [fill=black!100];
				\draw (1,1.6) circle (.7mm)  [fill=black!100];
				\node[left] at (0,2) {$n+1$};
				\node[right] at (3,2) {$\lightning$};
				\draw[dashed,opacity=0] (1,0) --(1,2.7);
				
				\draw[dashed,opacity=100] (2,.7) --(2,2.7);
				\draw (2,.7) circle (2mm)  [fill=white!100];
				\draw[dotted,opacity=1] (2,0) --(2,-.5);
				\node[left] at (2,-0.5) {$T$};
				\draw[dotted,opacity=1] (2.3,0) --(2.3,-.5);
				\node[right] at (2.3,-0.5) {$T-$};
				\end{tikzpicture}}
		\end{center}
	\end{minipage}
	\caption{The first event on the first~$n$ (out of at least~$n+1$) lines may be a branching (left), a pruning due to a deleterious mutation (center), or a pruning due to a beneficial mutation (right).}
	\label{fig:pldsasgfearnheadproof}
\end{figure}
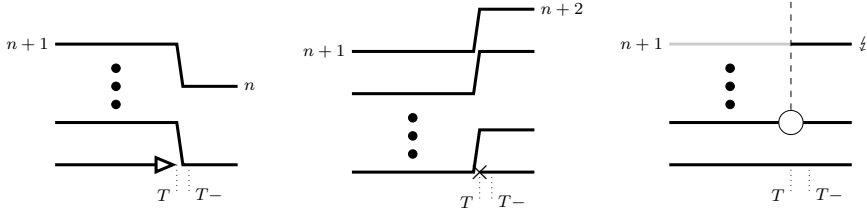
\begin{proof}[Proof of Proposition~\ref{prop:fearnheadrec}]
	We give a direct proof via the graphical construction (see Fig.~\ref{fig:pldsasgfearnheadproof}). The coefficients, as tail probabilities of a stationary distribution, satisfy the boundary conditions. Fix some~$r>0$ and~$n\in \N$. We now look at the last events before~$r$ in backward time; which correspond to the first events after~$r$ in forward time. Let~$T_s({r}), T_{\nu_0}({r})$, and~$T_{\nu_1}({r})$ be the (backward) times of the last selective, beneficial, and deleterious mutation event, respectively, that have occurred before time~${r}$ on the first~$n$ levels. Set~$T({r}):=\max\{T_s({r}),T_{\nu_0}({r}),T_{\nu_1}({r})\}$. On~$\{L_r>n\}$, we have that~$T(r)$ is positive. Furthermore, \begin{align*}
	P(L_{r}>n)=&\ P\big(L_{r}>n, T(r)=T_s(r)\big)+P\big(L_{r}>n, T(r)=T_{\nu_1}(r)\big)\\
	&+P\big(L_{r}>n, T(r)=T_{\nu_0}(r)\big).
	\end{align*}
	Let~$L_{T(r)-}:=\lim_{\tilde{r}\nearrow T(r)}L_{\tilde{r}}$ be the state `just before' the jump. Reading each transition in Fig.~\ref{fig:pldsasgfearnheadproof} from left to right, one concludes the following. If~$T(r)=T_s(r)$, then~$L_{r}>n$ if and only if~$L_{T(r)-}>n-1$. If~$T(r)=T_{\nu_1}(r)$, then~$L_{r}>n$ if and only if~$L_{T(r)-}>n+1$. The case~$T(r)=T_{\nu_0}(r)$ contradicts~$L_{r}>n$, so~$P\big(L_{r}>n,\ T(r)=T_{\nu_0}(r)\big)=0$. On~$\{L_{r}>n\}$, none of the first~$n$ lines is the immune line and therefore the probability that the last event is a selection event or a pruning due to a deleterious mutation is~$s/(u+s)$ and~$u\nu_1/(u+s)$, respectively. Hence,
	\begin{align*} P(L_{r}>n)=&\ \frac{s}{u+s}P\big(L_{T({r})-}>n-1\mid T({r})=T_s({r})\big) \\
	&+ \frac{u\nu_1}{u+s}P\big(L_{T({r})-}>n+1\mid T({r})=T_{\nu_1}({r})\big).\end{align*}
	But~$L_{T(r)-}$ is independent of what happens at time~$T(r)$, since this is in the future (in~$r$-time). Taking~${r}\to \infty$ on both sides proves the assertion.\qed
\end{proof}

If~$L$ is positive recurrent, we denote by~$L_{\infty}$ a random variable on~$\N$ distributed according to the stationary distribution of the line-counting process. Directly solving the recurrence relation \eqref{eq:fearnheadrecursion} leads to the geometric distribution of~$L_{\infty}$. Here, we take a different route. We derive the memoryless property of~$L_{\infty}$ and conclude that the distribution is geometric, since this is the only discrete distribution without memory. 
\begin{prop}[Lack of memory property] \label{prop:lackofmemory}
	If~$L$ is positive recurrent, then for all~$k\in \N_0$, \begin{equation} P(L_{\infty}>n+k \mid L_{\infty}>n)=P(L_{\infty}>k).\end{equation}
	In particular, \begin{equation}L_{\infty}\sim \Geom(1-p),\quad \text{with}\ p=\begin{cases}
	\frac{s}{u\nu_1}\bar{y}, &\text{if }\nu_1>0,\\
	\frac{s}{u+s},&\text{if }\nu_1=0.	\end{cases} \label{eq:stationarydistributionLinf} \end{equation}
\end{prop}
\begin{proof}
	Denote~$b_k^{(n)}:=P(L_{\infty}>n+k \mid L_{\infty}>n)$. Clearly,~$b_0^{(n)}=1$ and ${\lim_{k\to \infty} b_{k}^{(n)}=0}$ for all~$n\in \N$. By Proposition~\ref{prop:fearnheadrec}, $$b_k^{(n)}=\frac{a_{n+k}}{a_{n}}=\frac{1}{a_n}\Big(\frac{s}{u+s}a_{n+k-1}+\frac{u\nu_1}{u+s}a_{n+k+1}\Big)
	=\frac{s}{u+s}b_{k-1}^{(n)}+\frac{u\nu_1}{u+s}b_{k+1}^{(n)}.$$
	In particular,~$b_k^{(n)}=a_k$ for all~$k\in \N_0$. As a consequence,~$P(L_{\infty}>n)=a_1^n.$ Now that we know~$L_{\infty}$ has indeed a geometric distribution, it remains to determine the parameter. By Proposition~\ref{prop:fearnheadrec}, \begin{equation}
	a_1=\frac{s}{u+s}+ \frac{u\nu_1}{u+s}a_1^2,\label{eq:quadeqlt}
	\end{equation}
	of which the solution is given by~$a_1=s\bar{y}/u\nu_1$ if~$\nu_1>0$ and~$a_1=s/(u+s)$ if~$\nu_1=0$.\qed
\end{proof}

The recursion \eqref{eq:fearnheadrecursion} looks like a first-step decomposition for the absorption probabilities of some other process. And indeed, in the diffusion limit, \citet{baake2016} connect the tail probabilities of~$L_r$ to the absorption probabilities of another process via Siegmund duality. We establish a similar connection in the deterministic limit. Let~$(D_t)_{t\geq 0}$ be the process on~$\N^{\Delta}:=\N\cup \{\Delta\}$ with transition rates given by
\begin{equation}\label{eq:Drates}
q^{}_D(d,d-1)=(d-1)s, \qquad q^{}_D(d,d+1)=(d-1)u\nu_1, \qquad q^{}_D(d,\Delta)=(d-1)u\nu_0 
\end{equation}
for~$d \in \N$. We adopt the convention that~$n<\Delta$ for all~$n\in \N$.
\begin{remark}\label{rem:DRrelation}
	The process~$D$ exhibits an interesting connection to the line-counting process of the killed ASG. Let~$D$ and~$L$ be as previously defined with given rates~$u\nu_0,\ u\nu_1$, and~$s$. Furthermore, let~$\breve{D}$ and~$\breve{L}$ be the same processes, but with rates~$s$ and~$u\nu_1$ interchanged. Write~$\breve{R}$ for the line-counting process of a killed ASG with beneficial and deleterious mutation rate~$u\nu_0$ and~$s$, respectively, and selection rate~$u\nu_1$ (so~$u\nu_1$ and~$s$ are interchanged). Note that the rate at which mutation events occur is then~$\breve{u}=u\nu_0+s$; similarly, given that a mutation occurs, the probabilities for beneficial and deleterious mutations are~$\breve{\nu}_0=u\nu_0/(u\nu_0+s)$ and~$\breve{\nu}_1=s/(u\nu_0+s)$, respectively. Comparing \eqref{eq:Rrates} and \eqref{eq:Drates} immediately yields \begin{equation}
	\breve{R}\stackrel{d}{=}D-1 \quad \text{if}\quad \breve{R}_0 = D_0-1.\label{eq:Rplus1eqD}\end{equation} 
	In particular, the asymptotic behaviour of~$D$ follows by means of Lemma~\ref{lem:condabsorbR}: If~$\nu_0>0$,~$D$ absorbs in~$\{1,\Delta\}$ with probability~$1$. If in addition~$s=0$,~$D$ absorbs in~$\Delta$ with probability~$1$. If~$\nu_0=0$ and~$u\leq s$,~$D$ absorbs in~$1$ with probability~$1$. If~$\nu_0=0$ and~$u> s$,~$D$ absorbs in~$1$ with probability~$<1$ and, conditional on non-absorption of~$D$ in~$1$,~$D_t\to\infty$ with probability~$1$. 
\end{remark}

\begin{prop}\label{prop:siegmunddual}
	$L$ and~$D$ are Siegmund dual, i.e. for~$t\geq 0$, \begin{equation}
	P(m\leq L_t \mid L_0=n)=P(D_t\leq n \mid D_0=m), \qquad \forall n\in \N,\ m\in \N^{\Delta}. 
	\end{equation}
\end{prop}
\begin{proof}
	We denote the infinitesimal generators of~$L_t$ and~$D_t$ by~$\mathcal{A}_L$ and~$\mathcal{A}_D$, respectively. They have the form \begin{equation*}\begin{split}
	\mathcal{A}_Lf(n)=&\ ns [f(n+1)-f(n)]  +\left((n-1)u\nu_1+\1_{\{n\geq 1\}}u\nu_0\right)[f(n-1)-f(n)] \\
	&+ u\nu_0\sum_{i=1}^{n-2}[f(i)-f(n)]\end{split}
	\end{equation*}
	for~$f\in C_b(\N,\R)$, and 
	\begin{equation*}\begin{split}
	\mathcal{A}_D\tilde{f}(m)=&\ (m-1)s[\tilde{f}(m-1)-\tilde{f}(m)]+ (m-1)u\nu_1[\tilde{f}(m+1)-\tilde{f}(m)] \\
	&+ (m-1)u\nu_0[\tilde{f}(\Delta)-\tilde{f}(m)]\end{split}\label{eq:generatorD}\end{equation*}
	for~$\tilde{f}\in C_b(\N^{\Delta},\R).$ In the case of a Siegmund duality, the duality function is $$\bar{H}(n,m)=\1_{(m\leq n)}.$$ We will show that~$\mathcal{A}_L\bar{H}(\cdot,m)(n)=\mathcal{A}_D\bar{H}(n,\cdot)(m)$. The result follows then once more as an application of \citet[Thm.~3.42]{liggett2010}. Indeed,
	$$\mathcal{A}_L\bar{H}(\cdot,m)(n)=ns\1_{(m=n+1)}-(n-1)u\nu_1\1_{(m=n)}-u\nu_0 \sum_{i=1}^{n-1} \1_{(m\leq n<m+i)}.$$
	Note that $$\sum_{i=1}^{n-1} \1_{(m\leq n<m+i)} =\begin{cases}
	m-1, &\text{if }n\geq m> 1,\\
	0, &\text{otherwise}.
	\end{cases} $$
	Thus, we can rewrite~$\mathcal{A}_L\bar{H}(\cdot,m)(n)$ as$$(m-1)s\1_{(m-1=n)}-(m-1)u\nu_1\1_{(m=n)}-(m-1)u\nu_0\1_{(m\leq n)},$$
	which equals~$\mathcal{A}_D\bar{H}(n,\cdot)(m)$, as required.\qed
\end{proof}
\begin{remark}
	The analogous result in the diffusion limit is proven in~\citet{baake2016} via Clifford-Sudbury flights~\citep{clifford1985}. Their proof leads to a pathwise duality but relies on a particle representation of the forward process. We expect a similar argument to apply also in our setting, but as noted in Remark~\ref{rem:strongdualityforwardpicture}, this would require to first establish a particle representation for the forward process in the deterministic limit.
\end{remark}
\begin{coro}\label{coro:DtabsorbsLt} \begin{equation}\label{eq:auxiliaryabsorbtion}   P(\lim_{t\to\infty}D_t=1 \mid D_0=n+1)=a_n,\qquad \forall n\in \N.
	\end{equation}
\end{coro}\begin{proof} The proof is a direct consequence of Proposition~\ref{prop:siegmunddual}.\qed
\end{proof}
This result gives an alternative way to recover \eqref{eq:fearnheadrecursion} via a first-step decomposition of the absorption probabilities of~$D$. By Remark~\ref{rem:DRrelation}, the absorption probability of~$D$ in~$1$, given~$D_0=n+1$, equals the absorption probability of~$\breve{R}$ in~$0$, given~$\breve{R}_0=n$. In particular,~$a_n=\breve{\bar{y}}^n$, where \begin{equation}\label{eq:pvalueexp}
\breve{\bar{y}}:=\begin{cases}\frac{1}{2}\left(1+\frac{u\nu_0+s}{u\nu_1}-\sqrt{(1-\frac{u\nu_0+s}{u\nu_1})^2 + 4\frac{\nu_0}{\nu_1}}\, \right), & \nu_1>0,\\
\frac{s}{u+s}, &\nu_1=0.
\end{cases}\end{equation} This is consistent with Remark~\ref{rem:interchangesnu1}: The recursion for the absorption probability of~$\breve{R}$ in~$0$ is obtained by interchanging the roles of~$u\nu_1$ and~$s$ in Fearnhead's recursion. Hence,~$\breve{\bar{y}}$ is as in \eqref{eq:stablepoint}, but with~$u\nu_1$ and~$s$ interchanged; note that this implies replacement of~$u=u\nu_0+u\nu_1$ by~$u\nu_0+s$. On the other hand, as a consequence of Corollary~\ref{coro:DtabsorbsLt},~$\breve{\bar{y}}=p$ with~$p$ from Proposition~\ref{prop:lackofmemory}.
In a similar way, we can derive that $$\bar{y}=\lim_{r\to\infty} P_1(\breve{L}_{r}>1).$$
We can now deal with the asymptotic behaviour when~$L$ is null recurrent (recall from Proposition~\ref{prop:asymptoticLbeh} that this is the case for~$\nu_0=0$ and~$u=s$).
\begin{coro}\label{coro:nullrecurrentinfprob}
	If~$u=s$ and~$\nu_0=0$, $$\lim_{r\to\infty}P_1(L_r>n)=1.$$
\end{coro}
\begin{proof}
	The proof is an immediate consequence of Corollary~\ref{coro:DtabsorbsLt} together with Remark~\ref{rem:DRrelation}.\qed
\end{proof}

\section{The ancestral type distribution}\label{s6}
In the Moran model (with finite~$N$) and in the diffusion limit, all individuals at present originate from a single individual in the distant past, see \citet[Sect.~3]{Cordero2017590} and \citet[Thm.~3.2]{krone1997ancestral}, 
respectively. This individual is called the common ancestor, and the distribution of its type is the common ancestor type distribution. In the diffusion limit, \citet{fearnhead2002common} derived this distribution in terms of coefficients characterised via a recursion (see also \citet{Taylor2007}). In the deterministic limit, there are no coalescence and collision events \citep{Cordero2017590,krone1997ancestral}, so the notion of a common ancestor does not make sense. Instead \citet{Cordero2017590} introduces the representative ancestral type~(RA type). This is the type of the ancestor of a generic individual in the population, denoted earlier~\citep{georgii2003} as the ancestral type of a typical individual. The general concept was developed by \citet{JAGERS1989183,Jagers1992} in the context of branching processes. In our case, the RA type at backward time~$r$ is denoted by~$J_r$ and takes values in~$\{0,1\}$. For a representative ancestor that lives in a population with type distribution~$(1-y_0,y_0)$, we define~${g^{}_r(y_0):=P_{y_0}(J_r=1)}$ as the conditional probability of an unfit RA type at backward time~$r$. It follows from the discussion after Definition~\ref{def:pLDASG}~(alternatively, see \citet[Prop.~5.5, Cor.~5.6]{Cordero2017590}) that \begin{equation}
g^{}_r(y_0)=E_1[y_0^{L_r} ]=1-(1-y_0)\sum_{n\geq 0}P_1(L_{r}>n) y_0^n\label{eq:ratypetails}.\end{equation}
This is consistent with the graphical picture: the RA type at backward time~$r$ is~$1$ if and only if all~$L_r$ lines are of type~$1$. The corresponding probability is given by~$E_1[y_0^{L_r}]$. Alternatively, we can partition the event of a beneficial representative ancestor according to the first level occupied by a type-$0$ individual. Namely, $$P_1(L_r>n) (1-y_0)y_0^n $$is the probability that at least~$n+1$ lines are present, the~$(n+1)^{st}$ line is of type~$0$, and the first~$n$ lines are of type~$1$. Summing this probability over~$n$ gives the probability of an RA type~$0$. The complementary probability leads to the right-hand side of ~\eqref{eq:ratypetails}.

Now let~$g^{}_{\infty}(y_0):=\lim_{r\to\infty} g_r(y_0)$ be the conditional probability for an unfit RA type of an individual sampled at a very late time. If~${\sum_{n\geq 0}a_ny_0^n<\infty}$, equation \eqref{eq:ratypetails}~yields \begin{equation}
g^{}_{\infty}(y_0)=1-(1-y_0)\sum_{n\geq 0}a_ny_0^n.\label{eq:ratypetailslimit}\end{equation}
We can now exploit what we know about the~$a_n$ to obtain explicit expressions for~$g_{\infty}^{}$. This is captured in the following theorem.
\begin{thm}\label{thm:hy}
	\begin{enumerate}[label=(\roman*),leftmargin=25pt]
		\item If~$s=0$,~$g^{}_{\infty}(y_0)=y_0$ for all~$y_0\in [0,1].$
		\item If~$u\leq s$ and~$\nu_0=0$,~$g^{}_{\infty}(y_0)=\begin{cases} 0, &\text{if} \ y_0\in[0,1),\\
		1, &\text{if}\ y_0=1.\end{cases}$
		\item If~$s>0$ and either~$u>s$ or~$\nu_0>0$,~$g^{}_{\infty}(y_0)=\frac{1-p}{1-py_0}y_0.$
	\end{enumerate}
\end{thm}
\begin{proof} In case~$(i)$,~$L$ absorbs in~$1$ and hence~$a_0=1$ and~$a_n=0$ for all~$n\geq 1$. In particular,~$\sum_{n\geq 0}a_ny_0^n<\infty$ for all~$y_0\in [0,1]$, so that together with \eqref{eq:ratypetailslimit} the result follows. For case~$(ii)$, we first treat the subcase~$u<s$ and~$\nu_0=0$. There,~$L$ is transient and hence~$L_r\to\infty$ almost surely. Hence,~$a_n=1$ for all~$n\geq 0$. For~$y_0\in [0,1)$, again~$\sum_{n\geq 0}a_ny_0^n<\infty$, and the result follows by~\eqref{eq:ratypetailslimit}. For~$y_0=1$, we use that~$L$ is bounded for all~$r>0$. In particular, ${\sum_{n\geq 0}P(L_r>n\mid L_0=1)<\infty}$. But then,~$g^{}_{r}(1)=1$ for~$r>0$. Taking the limit~$r\to \infty$ yields the result. The other subcase of~$(ii)$ is~$u=s$ and~$\nu_0=0$. There, Corollary~\ref{coro:nullrecurrentinfprob} leads to~$a_n\equiv 1$~$(n\geq 0)$. Case~$(iii)$ follows by summing the geometric series obtained from \eqref{eq:ratypetailslimit} via Proposition~\ref{prop:lackofmemory}.\qed
\end{proof}

Theorem~\ref{thm:hy} is consistent with the graphical picture. Case~$(i)$ corresponds to the neutral situation, in which each individual has exactly one potential ancestor at all times; the representative ancestor is then a single draw from the initial distribution. In particular, there is no bias towards one of the types. In case~$(iii)$,~$L_{\infty}>1$ with positive probability, so there is a bias towards the beneficial type. The reason is that a single beneficial potential ancestor suffices for the RA type to be of type~$0$, which manifests itself in the factor~${(1-p)/(1-py_0)<1}$ for~$y_0<1$. In case~$(ii)$, depending on whether~$u=s$ or~$u<s$,~$L$ is null recurrent or transient. In both cases, the number of potential ancestors in the limit~$r\to \infty$ is infinite and the bias towards type~$0$ is taken to an extreme: Any positive proportion of beneficial types suffices to ensure that the ancestor has type~$0$. If there are no beneficial types in the population, the RA is of type~$1$ with probability~$1$.

Let us now study the dependence on~$\nu_0$ of the probability for an unfit RA type. To stress the dependence, we write~$g_{\infty}^{}(y_0,\nu_0)$.\begin{coro}\label{coro:monotonicitynu0}
	Let~$y_0\in (0,1)$. If~$s>0$ and~$\nu_0,\bar{\nu}_0\in [0,1]$ with~$\nu_0<\bar{\nu}_0$, $$g_{\infty}^{}(y_0,\nu_0)<g_{\infty}^{}(y_0,\bar{\nu}_0).$$
\end{coro}
\begin{proof}
	Fix~$s,u>0$,~$y_0\in (0,1)$, and~$\nu_0,\bar{\nu}_0\in[0,1]$ with~$\nu_0<\bar{\nu}_0$. Furthermore, we set~${\nu_1=1-\nu_0}$ and~$\bar{\nu}_1=1-\bar{\nu}_0$. Clearly,~$\nu_1>\bar{\nu}_1$. We use a coupling argument similar to \citet[Sect.~6]{Lenz201527}. Write~${\Gb=(\Gb_{r})_{r\geq 0}}$ for the pLD-ASG with selection rate~$s$, mutation rate~$u$, beneficial mutation probability~$\nu_0$, and deleterious mutation probability~$\nu_1$. Let~$L=(L_{r})_{r\geq 0}$ be the line-counting process of~$\Gb$. We write~$\bar{\Gb}=(\bar{\Gb}_{r})_{r\geq 0}$ for another pLD-ASG, with line-counting process
	$\bar{L}=(\bar{L}_{r})_{r\geq 0}$, which we couple with~$\Gb$ in such a way that~$\bar{L}_r\leq L_r$ almost surely, for all~$r\geq 0$. We start with~$L_0=\bar{L}_0=1$. Assume that we have constructed~$\{\Gb_q\}_{q<r}$ and~$\{\bar{\Gb}_q\}_{q<r}$, and that~$\bar{L}_{r-}\leq L_{r-}$ almost surely. If a line of~$\Gb_{r-}$ at level~$i\leq \bar{L}_{r-}$ branches at time~$r$, then also the line of~$\bar{\Gb}_{r-}$ at level~$i$ branches. If on a line of~$\Gb_{r-}$ at level~$i\leq \bar{L}_{r-}$ there is a beneficial mutation at time~$r$, then the line of~$\bar{\Gb}_{r-}$ at level~$i$ has a beneficial mutation. 
	If on a line of~$\Gb_{r-}$ at level~$i\leq L_{r-}$ there is a deleterious mutation at time~$r$, we toss a coin. With probability~$\bar{\nu}_1/\nu_1$ we have a deleterious mutation on the line of~$\bar{\Gb}_{r-}$ at level~$i$; but with probability~$1-\bar{\nu}_1/\nu_1$ we have a beneficial mutation on the line of~$\bar{\Gb}_{r-}$ at level~$i$. In all the cases~$\bar{L}_r\leq L_r$ almost surely. Constructing~$\bar{\Gb}$ in this inductive manner leads to
	a pLD-ASG with selection rate~$s$, mutation rate~$u$, beneficial mutation probability~$\bar{\nu}_0$, and deleterious mutation probability~$\bar{\nu}_1$ with the desired property. In particular, using \eqref{eq:ratypetails}, we get
	$$g_r(y_0,\nu_0)=E_1\big [y_0^{L_{r}}\big]\leq E_1\big[y_0^{\bar{L}_{r}}\big]=g_r(y_0,\bar{\nu}_0).$$
	Letting~$r\to\infty$ leads to~$g_{\infty}^{}(y_0,\nu_0)\leq g_{\infty}^{}(y_0,\bar{\nu}_0).$ In order to show that the inequality holds in the strict sense, it is enough to show that~${P_1\big(\bar{L}_\infty<L_\infty\big)>0}$.
	To see this, we let~$T(r)$ be the time of the last deleterious mutation on the lowest line that has occurred before time~$r$ in~$L$~(with the convention~${T(r):=-\infty}$ if there is no such mutation). On the set~${\{T(r)>0\}}$, let $L_{T(r)-}:= \lim_{q\nearrow T(r)}L_{q}$ be the state `just before' the mutation. 
	Note that 
	$$P_1\big(\bar{L}_r<L_r\big)\geq P_1\big(\bar{L}_r<L_r,T(r)>0\big)\geq \left(1-\frac{\bar{\nu}_1}{\nu_1}\right)P_1(L_{T(r)-}>1,T(r)>0)$$ for~$r>0$. Clearly~$T(r)$ tends to~$\infty$ as~$r\to\infty$. Hence, we also have that ${\lim_{r\to\infty}P_1(L_{T(r)-}>1,T(r)>0)=p}$, where~$p$ is given in Proposition~\ref{prop:asymptoticLbeh}. In particular,~$P_1\big(\bar{L}_{\infty}<L_{\infty}\big)\geq p\big(1-\bar{\nu}_1/\nu_1\big)>0$, and the result follows.\qed
\end{proof}

Consider now~$(1-g_{\infty}^{}(\bar{y}),g_{\infty}^{}(\bar{y}))$, namely, the distribution of the RA type that lives in the stable equilibrium population~$\bar{y}$. We call it the RA type distribution at equilibrium and characterise it in what follows. First note that both~$g_{\infty}^{}$ and~$\bar{y}$ are functions of~$\nu_0$. In order to stress this~(double) dependence, we write~$g_{\infty}^{}(\bar{y}(\nu_0),\nu_0)$. We also write~$L(\nu_0)$ instead of~$L$. Recall from Section~\ref{sec:moran} that, for~$\nu_0=0$,~$\bar{y}(0)=\min\{u/s,1\}$ and so, by Theorem~\ref{thm:hy}, 
$$g^{}_{\infty}(\bar{y}(0),0)=\begin{cases}0,&\text{if }u<s,\\
1,&\text{if }u\geq s.\end{cases} $$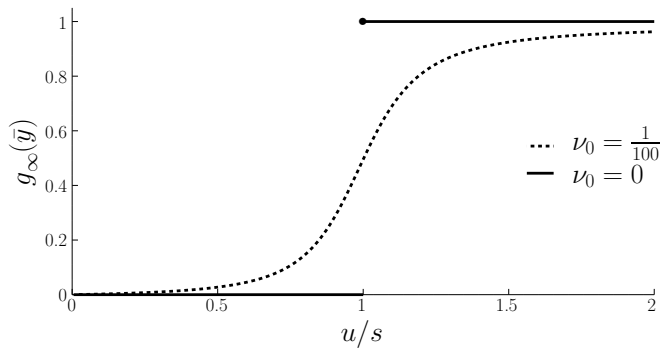
\begin{figure}[t!]
	\begin{center}
		\scalebox{.35}{
			\input{ancestraldistribution-fig.tex}
		}
	\end{center}
	\caption{The probability of an unfit RA type at equilibrium.}
	\label{fig:ratdistequi}
\end{figure}
This is the counterpart to the transcritical bifurcation~(or error threshold) of the equilibrium frequency of the forward process. The probability for an unfit RA type at equilibrium exhibits an even more drastic behaviour: a jump from~$0$ to~$1$ if~$u$ surpasses the critical value~$s$, see Fig.~\ref{fig:ratdistequi}. If~${\nu_0\in (0,1)}$,~$L(\nu_0)$~is positive recurrent and~$L_\infty(\nu_0)$ is almost surely finite. Moreover, in this case~${\bar{y}(\nu_0)\in(0,1)}$ for all~$u>0$. In particular,~$g_{\infty}^{}(\bar{y}(\nu_0),\nu_0)=E_1 [{\bar{y}(\nu_0)}^{L_\infty(\nu_0)}]\in (0,1)$, and hence
$$g^{}_{\infty}(\bar{y}(\nu_0),\nu_0)\begin{cases}
>g^{}_{\infty}(\bar{y}(0),0)=0,&\text{if }u<s,\\
< g^{}_{\infty}(\bar{y}(0),0)=1,&\text{if }u\geq s,
\end{cases} $$ compare Figs.~\ref{fig:stationarydistribution} and~\ref{fig:ratdistequi}. It may seem surprising at first sight that, even though switching off beneficial mutations leads to an increase of~$\bar{y}$ for all values of~$u$, it decreases the probability for the deleterious type to be ancestral if~$u<s$, but increases it for~$u\geq s$. The reason for this is that:
\begin{itemize}
	\item for~$u<s$: in contrast to the case~$\nu_0\in (0,1)$, where~$L_\infty(\nu_0)$ is finite,~$L_\infty(0)$ is infinite, and hence beats~$\bar{y}(0)$ regardless of its value. 
	\item for~$u\geq s$: in contrast to the case~$\nu_0\in (0,1)$, where~$\bar{y}(\nu_0)$ is strictly positive,~$\bar{y}(0)=1$, and therefore, there is no chance to sample an ancestor of type~$0$, regardless of the value of~$L_\infty(\nu_0)$. 
\end{itemize}

Let us finally give an alternative representation of the conditional probability of an unfit RA-type. Motivated by \citet{Taylor2007} and \citet[Sect.~7]{baake2016} in the diffusion limit, we consider the piecewise-deterministic Markov process~$\tilde{Y}:=(\tilde{Y}_t)_{t\geq 0}$ with generator 
\begin{equation}\label{eq:jprocess}\begin{split}
\mathcal{A}_{\tilde{Y}}f(y)= &\ [-sy(1-y)\!-\!u\nu_0y\!+\!u\nu_1(1\!-\!y)]\frac{\partial f}{\partial y} +\frac{u\nu_0y}{1-y}\left[f(1)\!-\!f(y) \right]\\
&+\frac{u\nu_1(1\!-\!y)}{y}\left[f(0)\!-\!f(y)\right]
\end{split} \end{equation}
for~$f\in  C_1([0,1],\R)$ with~$\lim_{y\to 1}\mathcal{A}_{\tilde{Y}}f(y)=\lim_{y\to 0}\mathcal{A}_{\tilde{Y}}f(y)=0$. The latter means that~$\tilde{Y}$ absorbs in~$0$ or~$1$. This process follows the dynamics of the mutation-selection equation up to a random jumping time. At this time the process jumps to one of the boundaries where it is absorbed. Existence and uniqueness of a Markov process corresponding to~$\mathcal{A}_{\tilde{Y}}$ follow by proving that the jump rates, which diverge at the boundary, are in fact bounded along trajectories of the process over any finite time interval. If~$\nu_0\in (0,1)$ and $y_0\in (0,1)$, then~$y(t;y_0)$ never hits the boundary, since $\bar{y}\in (0,1)$. If~$\nu_0=1$ or~$\nu_0=0$ and~${y_0\in (0,1)}$, then~$y(t;y_0)$ hits~$0$ or~$1$, respectively. In both cases, the possibly diverging jump term is absent because either~$\nu_1=0$ or~$\nu_0=0$, respectively. We now show that~$\tilde{Y}$ is dual to~$L$.
\begin{thm} \label{thm:determ.dual.mut.Lt}
	The processes~$\tilde{Y}$ and~$L$ are dual with respect to the duality function~$H(y,n)=y^n$~(from \eqref{eq:dualityfct}), that is, for~$t\geq0$,
	\begin{equation}\label{eq:dualitytildeyL}
	E\big[\big(\tilde{Y}_t\big)^n \mid \tilde{Y_0}=y_0 \big]=E\big[y_0^{L_t} \mid L_0=n \bigr], \qquad \forall y_0\in[0,1], \ n\in \N.
	\end{equation}
\end{thm}
\begin{proof} Once more we apply the duality criterion for generators \citep[Thm.~3.42]{liggett2010}. Note that we can rewrite~$\mathcal{A}_{\tilde{Y}}=\mathcal{A}_y^s+\mathcal{A}^{u}_{\tilde{Y}},$ with~$\mathcal{A}_y^s$ of \eqref{eq:generatorYparts} and \begin{align*}\mathcal{A}^{u}_{\tilde{Y}}f(y):=& \ [-yu\nu_0+u\nu_1(1-y)]\frac{\partial f}{\partial y} +\frac{y}{1-y}u\nu_0\left[f(1)-f(y) \right]\\
	&+\frac{1-y}{y}u\nu_1\left[f(0)-f(y)\right],
	\end{align*}
	which should not be confused with~$\mathcal{A}_Y^u$ \eqref{eq:generatorYparts}. In a similar way,~$\mathcal{A}_L=\mathcal{A}_R^s+\mathcal{A}_L^{u},$ with~$\mathcal{A}_R^s$ of \eqref{eq:generatorRparts} and $$\mathcal{A}_L^{u}f(n)=(n-1)u\nu_1[f(n-1)-f(n)]+u\nu_0\sum_{k=1}^{n-1}[f(k)-f(n)].$$
	In the proof of Theorem~\ref{thm:dualityRY}, we already showed~$\mathcal{A}_y^sH(\cdot,n)(y)=\mathcal{A}_R^sH(y,\cdot)(n).$ Hence, it suffices to check~$\mathcal{A}_{\tilde{Y}}^{u}H(\cdot,n)(y)=\mathcal{A}_L^uH(y,\cdot)(n)$, which then implies $$\mathcal{A}_{\tilde{Y}} H(\cdot,n)(y)=\mathcal{A}_LH(y,\cdot)(n),\qquad \forall y\in [0,1], n\in \N.$$ Indeed, we obtain $$\begin{aligned}
	&\ \mathcal{A}_{\tilde{Y}}^{u} H(\cdot,n)(y)\\
	=&\ (n-1)u\nu_1\left[y^{n-1}-y^n\right]+ u\nu_0\left[y \frac{1-y^n}{1-y}-ny^n\right] \\
	=&\ (n-1)u\nu_1\left[H(y,n-1)-H(y,n)\right]+u\nu_0\sum_{j=1}^{n-1}\left[H(y,j)-H(y,n) \right]\\
	=&\ \mathcal{A}_L^uH(y,\cdot)(n),
	\end{aligned} $$ which proves the claim.\qed
\end{proof}

We now obtain a characterisation of~$g_{\infty}^{}(y_0)$ that does not depend on~$L$ by taking~$t\to\infty$ in~\eqref{eq:ratypetails} and~\eqref{eq:dualitytildeyL}.
\begin{coro} \label{coro:absorptiontildey}For~$y_0\in[0,1]$, we have
	$$g^{}_{\infty}(y_0)=P(\lim_{t\to\infty} \tilde{Y}_t = 1 \mid  \tilde{Y}_0=y_0).$$
\end{coro}
\begin{remark} The Kolmogorov backward equation for the absorption probability of~$\tilde{Y}_t$ in~$1$ leads to the characterisation of~$g_{\infty}^{}$ as the solution to the boundary value problem \begin{equation} \label{eq:bvptildey}\mathcal{A}_{\tilde{Y}}g_{\infty}^{}(y_0)=0\quad  \text{for }y_0\in (0,1),\end{equation} complemented by~$g_\infty^{}(0)=0$ and~$g_{\infty}^{}(1)=1$. It is the deterministic limit analogue of the boundary value problem in \citet[Eqn.~(11)]{Taylor2007}.
\end{remark}
\begin{figure}[b!]
	\begin{minipage}{0.5\textwidth} 
		\begin{center}
			\scalebox{.6}{\begin{tikzpicture}
				\node [left] at (0,0.7) {\scalebox{1.5}{$1$}};
				\node [left] at (0,0) {\scalebox{1.5}{$1$}};
				\draw[line width=.5mm ] (3,2.1) -- (4,2.1) -- (4.1,1.4) -- (5,1.4) -- (5.1,.7) -- (6,.7) -- (6.1,0);
				\draw[line width=1mm,dotted] (6.1,0) -- (7,0) -- (8,0);
				\draw[line width=1mm] (0,0.7)-- (0.9,0.7) -- (1,1.4) -- (2,1.4) -- (2.1,.7) -- (3,0.7) ;
				\draw[-{open triangle 45[scale=2.5]},line width=1mm,dotted] (3,.7) -- (4,.7) -- (4.1,0) -- (6,0);
				\draw[-{open triangle 45[scale=2.5]},line width=.5mm] (3,1.4) -- (4,1.4) -- (4.1,.7) -- (5,.7);
				\draw[-{open triangle 45[scale=2.5]},line width=.5mm] (1,0) -- (4,0);
				\draw[-{open triangle 45[scale=2.5]},line width=.5mm] (0,0 )-- (0.9,0) -- (1,0.7) -- (2,0.7);
				\draw[-{>},line width=0.5mm] (3.5,-.3) -- (3.5,-1);
				\draw[dashed,opacity=100] (3,.7) --(3,2.45);
				\draw (3,.7) circle (1.7mm)  [fill=white!100];
				\draw (7,0) circle (1.7mm)  [fill=white!100];
				\node at (1,0) {\scalebox{1.6}{$\times$}} ;
				
				\draw[line width=1mm,owngrey] (0,-1.2) -- (3,-1.2);
				\draw[line width=1mm,dotted,owngrey] (3,-1.2) -- (8,-1.2);
				\draw (3,-1.2) circle (1.7mm)  [fill=white!100];
				\end{tikzpicture}}
		\end{center}
	\end{minipage}
	\begin{minipage}{0.4\textwidth} 
		\begin{center}
			\scalebox{.6}{\begin{tikzpicture}
				\node [left] at (-.2,0.7) {\scalebox{1.5}{$1$}};
				\node [left] at (-.2,0) {\scalebox{1.5}{$0$}};
				\draw[line width=.5mm ] (3,2.1) -- (4,2.1) -- (4.1,1.4) -- (5,1.4) -- (5.1,.7) -- (6,.7) -- (6.1,0);
				\draw[line width=.5mm ] (0,0.7)-- (0.9,0.7) -- (1,1.4) -- (2,1.4) -- (2.1,.7);
				\draw[-{open triangle 45[scale=2.5]},line width=1mm,dotted] (2.1,.7) -- (3,.7) -- (4,.7) -- (4.1,0) -- (6,0);
				\draw[-{open triangle 45[scale=2.5]},line width=.5mm] (3,1.4) -- (4,1.4) -- (4.1,.7) -- (5,.7);
				\draw[-{open triangle 45[scale=2.5]},line width=.5mm] (1,0) -- (4,0);
				\draw[-{open triangle 45[scale=2.5]},line width=1mm,dotted] (0,0 )-- (0.9,0) -- (1,0.7) -- (2,0.7);
				\draw[dashed,opacity=100] (3,.7) --(3,2.45);
				\draw[-{>},line width=0.5mm] (3.5,-.3) -- (3.5,-1);
				\draw (3,.7) circle (1.7mm)  [fill=white!100];
				\node at (1,0) {\scalebox{1.6}{$\times$}} ;
				\draw[line width=1mm,dotted] (6.1,0) -- (8,0);
				\draw (7,0) circle (1.7mm)  [fill=white!100];
				\draw[line width=1mm,dotted,owngrey] (0,-1.2) -- (8,-1.2);
				\end{tikzpicture}}
		\end{center}
	\end{minipage} \hfill
	\caption{The type along the ancestral line depends on the pLD-ASG and on the initial assignment of types. A solid~(dotted) ancestral line corresponds to an unfit~(fit) type. The type assignment on the left leads to a type change due to a beneficial mutation. The type assignment on the right does not lead to a type change on the ancestral line. The extracted type evolution along the ancestral line is depicted in grey below.}
	\label{fig:pLDASGtrueancestralline}
	\end{figure}
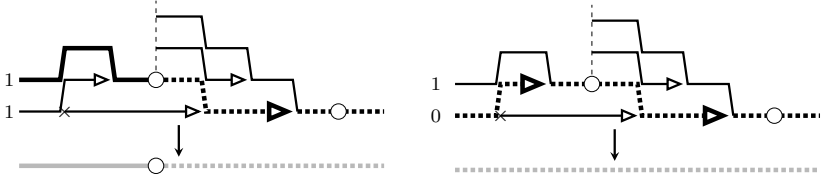
	
In the remainder of the paper we provide a heuristic explanation for the duality relation of Theorem~\ref{thm:determ.dual.mut.Lt}. For simplicity, we only treat the case~${n=1}$ here, but the argument is easily adapted to arbitrary~$n$. We sample an individual at time~$t$ and construct the pLD-ASG until time~$0$ where the type distribution is~$(1-y_0,y_0)$. Let~$L_t$ be the number of lines in the pLD-ASG at time~$0$. Then, $$y_0^{L_t}=P_{y_0}(J_t=1\mid L_t).$$
Our aim is to construct an appropriate random variable~$K_t$~(different from~$L_t$), such that $$\tilde{Y}_t=P_{y_0}(J_{t}=1\mid K_t).$$
Theorem~\ref{thm:determ.dual.mut.Lt} would then be a direct consequence of this result. Note that for every realisation of the pLD-ASG and each initial configuration, we can extract the type evolution along the ancestral line and partition it into pieces in which the type is constant; see~Fig.~\ref{fig:pLDASGtrueancestralline}. For a given pLD-ASG and a given assignment of types, the type along the ancestral line changes only by mutations. We will keep track of the kind of mutation that induces the last type change~(that is, the last non-silent mutation) along the ancestral line before time~$0$. Let~$K_t$ be the random variable with values in~$\{\varnothing,\Circle, \times\}$ that encodes this. The symbols~$\Circle$ or~$\times$ represent a beneficial or deleterious type-changing mutation; the symbol~$\varnothing$ indicates that there is no type change in~$[0,t]$. Clearly,~$K_t$ is a function of the pLD-ASG and the initial assignment of types. If there is no type change on the ancestral line, i.e.~$\{K_t=\varnothing\}$, the type of the sampled individual and its ancestor at time~$t$ coincide, and therefore this type is~$1$ with probability~$y(t;y_0)$. In particular, before we see a type-changing mutation on the ancestral line, the desired probability evolves as the proportion of deleterious individuals in the population, which explains the drift term in the generator of the process~$\tilde{Y}$. The event~$\{K_t=\Circle\}$ implies that the type of the ancestor at time~$0$ is deleterious and therefore~$\tilde{Y}$ must be~$1$. This explains the jump to~$1$ of the process~$\tilde{Y}$. The event~$\{K_t=\times\}$ implies that the ancestor at time~$0$ is of the beneficial type, and this explains the jump to~$0$ of the process~$\tilde{Y}$. Summarising, $$P_{y_0}(J_t=1\mid K_t)=\begin{cases}
y(t;y_0),&\text{if }K_t=\varnothing,\\
1,&\text{if }K_t=\Circle,\\
0,&\text{if }K_t=\times,
\end{cases}$$ see also Fig.~\ref{fig:typechangerate}. It remains to explain the rates at which these type changes appear on the ancestral line. We only explain the jump rate to~$0$; the jump rate to~$1$ follows in an analogous way. The rate at which a deleterious type change at time~$t$ occurs on the ancestral line, and therefore a jump to~$0$ of~$\tilde{Y}$, is given by~$\lim_{\varepsilon\rightarrow 0}P_{y_t}^{\varepsilon}(J_\varepsilon=0 \mid J_0=1)/\varepsilon$, where~$P_{y_t}^\varepsilon$ means that we condition on the type distribution at backward time~$\varepsilon$ being equal to~$(1-y(t;y_0),y(t;y_0))$. Note that 
\begin{figure}[t!]
	\begin{minipage}{0.44\textwidth} 
		\begin{center}
			\scalebox{.6}{\begin{tikzpicture}
				\draw[dashed] (3.2,-1.3) --(3.2,2.2);
				
				
				\draw[line width=.2mm] (0,0) -- (0,4);
				\draw[line width=.2mm] (0,0) -- (8,0);
				\node [right] at (0,-.5) {\scalebox{1.5}{$0$}};
				\node [left] at (8,-.5) {\scalebox{1.5}{$t$}};
				\node [right] at (-.5,.2) {\scalebox{1.5}{$0$}};
				\node [right] at (-.5,3.8) {\scalebox{1.5}{$1$}};
				
				\node [right] at (1.3,2.5) {\scalebox{1.65}{$\tilde{Y}_t$}};
				\node [right] at (5.5,2) {\scalebox{1.65}{$y(t;y_0)$}};
				
				\draw[line width=.3mm] (0,1) to [out= 32, in=180] (8,2.5);
				\draw[line width=1.5mm] (0,1) to [out= 32, in=191] (3.2,2.28);
				\draw[line width=1.5mm] (8,4) -- (3.2,4);
				\draw[dashed,opacity=0] (2,0) --(2,4);
				\draw[line width=1mm,gray] (0,-1.5) -- (3.2,-1.5);
				\draw[line width=1mm,gray,dotted] (3.2,-1.5) -- (8,-1.5);
				\draw (3.2,-1.5) circle (1.7mm)  [fill=white!100];
				\end{tikzpicture}}
		\end{center}
	\end{minipage}
	\hfill
	\begin{minipage}{0.44\textwidth} 
		
		\begin{center}
			\scalebox{.6}{\begin{tikzpicture}
				\draw[dashed] (3.2,-1.3) --(3.2,2.2);
				
				
				\draw[line width=.2mm] (0,0) -- (0,4);
				\draw[line width=.2mm] (0,0) -- (8,0);
				\node [right] at (0,-.5) {\scalebox{1.5}{$0$}};
				\node [left] at (8,-.5) {\scalebox{1.5}{$t$}};
				\node [right] at (-.5,.2) {\scalebox{1.5}{$0$}};
				\node [right] at (-.5,3.8) {\scalebox{1.5}{$1$}};
				
				\node [right] at (1.3,2.5) {\scalebox{1.65}{$\tilde{Y}_t$}};
				\node [right] at (5.5,2) {\scalebox{1.65}{$y(t;y_0)$}};
				
				\draw[line width=.3mm] (0,1) to [out= 32, in=180] (8,2.5);
				\draw[line width=1.5mm] (0,1) to [out= 32, in=191] (3.2,2.28);
				\draw[line width=1.5mm] (8,0) -- (3.2,0);
				\draw[dashed,opacity=0] (2,0) --(2,4);
				\draw[line width=1mm,owngrey,dotted] (0,-1.5) -- (3.2,-1.5);
				\draw[line width=1mm,owngrey] (3.2,-1.5) -- (8,-1.5);
				\node at (3.2,-1.5) {\scalebox{1.6}{$\times$}} ;
				\end{tikzpicture}}
		\end{center}
	\end{minipage}
\caption{A type change on the ancestral line~(grey) can either be due to a mutation to type~$0$ or type~$1$. The rate for either one depends on the type distribution at the time of the type change. The bold solid line corresponds to~$\tilde{Y}_t$, whereas the thin solid line corresponds to~$y(t;y_0)$.}
	\label{fig:typechangerate}
	
\end{figure}
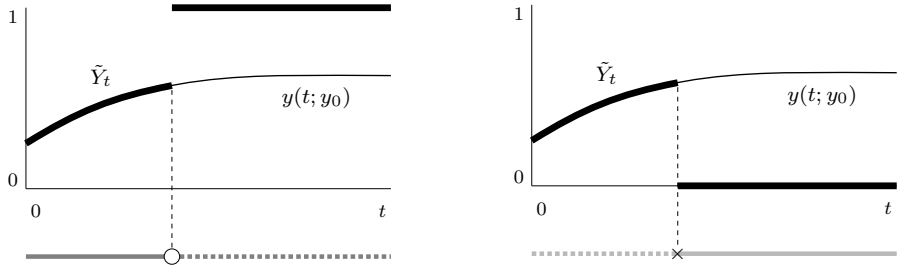 
\begin{equation}\begin{split}
P_{y_t}^\varepsilon(J_{\varepsilon}=0\mid J_{0}=1)=&\ P_{y_t}^\varepsilon(J_{0}=1\mid J_{\varepsilon}=0)\frac{P_{y_t}^\varepsilon(J_{\varepsilon}=0)}{P_{y_t}^\varepsilon(J_{0}=1)}\\
=&\ P_{y_t}^\varepsilon(J_{0}=1\mid J_{\varepsilon}=0)\,\frac{1-E_1[y(t;y_0)^{L_{\varepsilon}}]}{y(t+\varepsilon,y_0)}.\end{split}\label{eq:Jr0J01}
\end{equation} 
Moreover, denoting by~$M_\varepsilon^1$ the event of a single deleterious~(not necessarily type-changing) mutation on the ancestral line in a time interval of length~$\varepsilon$, we deduce that
\begin{equation}\begin{aligned}
\label{eq:J01Jr0} P_{y_t}^\varepsilon(J_{0}=1\mid J_{\varepsilon}=0)&=P_{y_t}^\varepsilon(J_{0}=1,M^1_{\varepsilon}\mid J_{\varepsilon}=0)+o(\varepsilon)\\
&=P_{y_t}^\varepsilon(M^1_{\varepsilon}\mid J_{\varepsilon}=0)+o(\varepsilon)\\
&=u\nu_1\varepsilon+o(\varepsilon).\end{aligned}\end{equation}
Combining \eqref{eq:Jr0J01} and \eqref{eq:J01Jr0} leads to $$\lim\limits_{\varepsilon\rightarrow 0}\frac{P_{y_t}(J_{\varepsilon}=0\mid J_{0}=1)}{\varepsilon}=u\nu_1\frac{1-y(t;y_0)}{y(t;y_0)},$$
which corresponds to the jump rate to~$0$ in \eqref{eq:jprocess}.
\begin{remark}
	It seems this argument can be applied to obtain a pathwise duality. But as noted in Remark~\ref{rem:strongdualityforwardpicture}, this requires a particle representation of the forward process in the deterministic limit.
\end{remark}


\begin{acknowledgements}
It is our pleasure to thank Anton Wakolbinger and Ute Lenz for stimulating and fruitful discussions. Furthermore, we are grateful to Anton Wakolbinger, Jay Taylor, and an unknown referee for helpful comments on the manuscript. This project received financial support from Deutsche Forschungsgemeinschaft~(Priority Programme SPP 1590 ‘Probabilistic Structures in Evolution’, Grant No. BA 2469/5-1, and CRC 1283 ‘Taming Uncertainty’, Project~C1).
\end{acknowledgements}



%

\end{document}

%% file: stationarydistribution-fig.tex
\setlength{\unitlength}{1pt}
\begin{picture}(0,0)
\includegraphics{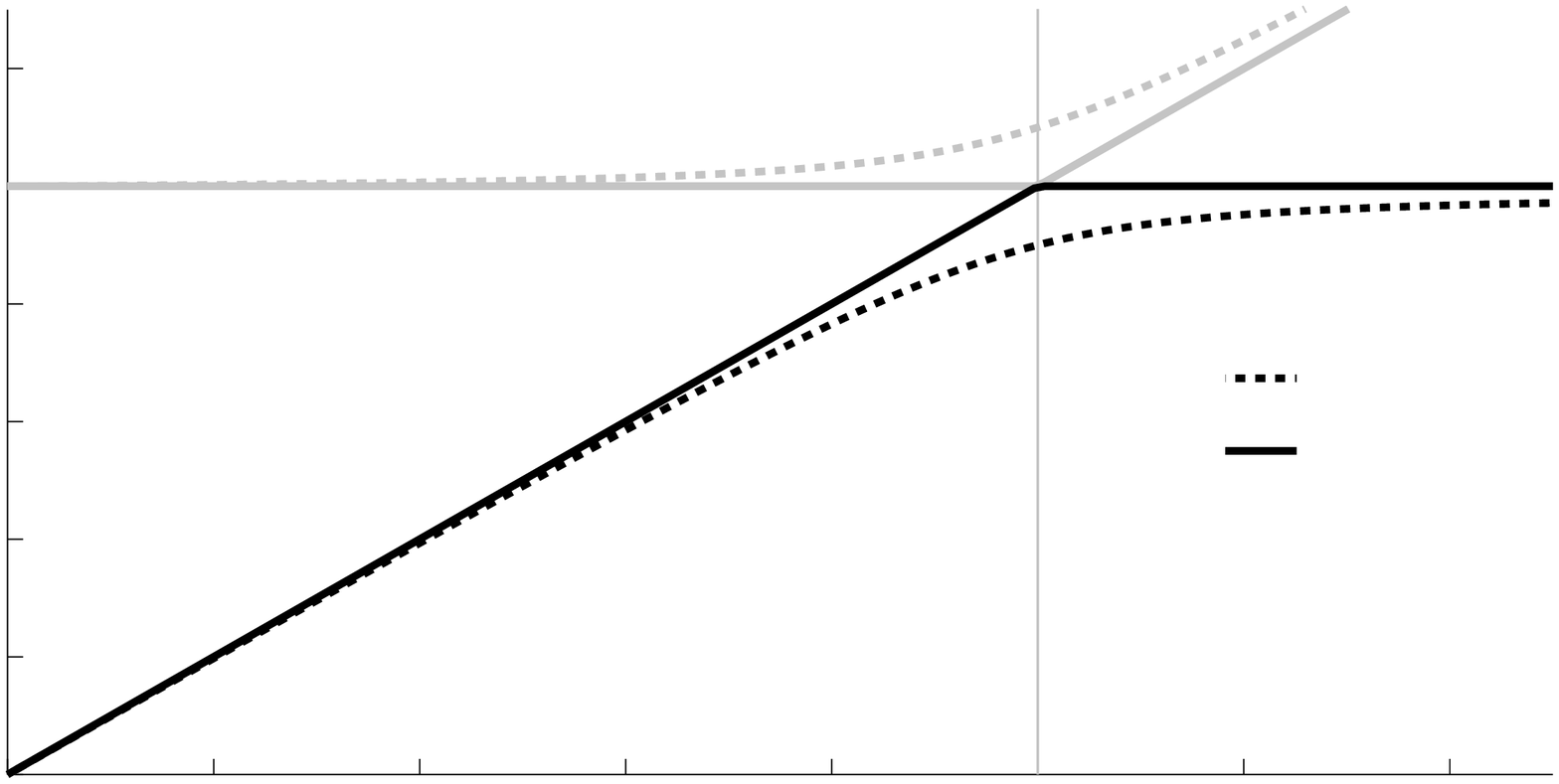}
\end{picture}%
\begin{picture}(800,401)(0,0)
\fontsize{20}{0}
\selectfont\put(104,58.0014){\makebox(0,0)[t]{\textcolor[rgb]{0,0,0}{{0}}}}
\fontsize{20}{0}
\selectfont\put(186.667,58.0014){\makebox(0,0)[t]{\textcolor[rgb]{0,0,0}{{0.2}}}}
\fontsize{20}{0}
\selectfont\put(269.333,58.0014){\makebox(0,0)[t]{\textcolor[rgb]{0,0,0}{{0.4}}}}
\fontsize{20}{0}
\selectfont\put(352,58.0014){\makebox(0,0)[t]{\textcolor[rgb]{0,0,0}{{0.6}}}}
\fontsize{20}{0}
\selectfont\put(434.667,58.0014){\makebox(0,0)[t]{\textcolor[rgb]{0,0,0}{{0.8}}}}
\fontsize{20}{0}
\selectfont\put(517.333,58.0014){\makebox(0,0)[t]{\textcolor[rgb]{0,0,0}{{1}}}}
\fontsize{20}{0}
\selectfont\put(600,58.0014){\makebox(0,0)[t]{\textcolor[rgb]{0,0,0}{{1.2}}}}
\fontsize{20}{0}
\selectfont\put(682.667,58.0014){\makebox(0,0)[t]{\textcolor[rgb]{0,0,0}{{1.4}}}}
\fontsize{20}{0}
\selectfont\put(99,63.0022){\makebox(0,0)[r]{\textcolor[rgb]{0,0,0}{{0}}}}
\fontsize{20}{0}
\selectfont\put(99,110.24){\makebox(0,0)[r]{\textcolor[rgb]{0,0,0}{{0.2}}}}
\fontsize{20}{0}
\selectfont\put(99,157.478){\makebox(0,0)[r]{\textcolor[rgb]{0,0,0}{{0.4}}}}
\fontsize{20}{0}
\selectfont\put(99,204.717){\makebox(0,0)[r]{\textcolor[rgb]{0,0,0}{{0.6}}}}
\fontsize{20}{0}
\selectfont\put(99,251.955){\makebox(0,0)[r]{\textcolor[rgb]{0,0,0}{{0.8}}}}
\fontsize{20}{0}
\selectfont\put(99,299.193){\makebox(0,0)[r]{\textcolor[rgb]{0,0,0}{{1}}}}
\fontsize{20}{0}
\selectfont\put(99,346.431){\makebox(0,0)[r]{\textcolor[rgb]{0,0,0}{{1.2}}}}
\fontsize{30}{0}
\selectfont\put(414,36.0014){\makebox(0,0)[t]{\textcolor[rgb]{0,0,0}{{$u/s$}}}}
\fontsize{30}{0}
\selectfont\put(67,216.526){\rotatebox{90}{\makebox(0,0)[b]{\textcolor[rgb]{0,0,0}{{$\bar{y},y^{\star}$}}}}}
\fontsize{30}{0}
\selectfont\put(624.957,225.1){\makebox(0,0)[l]{\textcolor[rgb]{0,0,0}{{$\ \ \nu_0=\frac{1}{100}$}}}}
\fontsize{30}{0}
\selectfont\put(624.957,190.941){\makebox(0,0)[l]{\textcolor[rgb]{0,0,0}{{$\ \ \nu_0=0$}}}}
\end{picture}

%% file: ancestraldistribution-fig.tex
\setlength{\unitlength}{1pt}
\begin{picture}(0,0)
\includegraphics{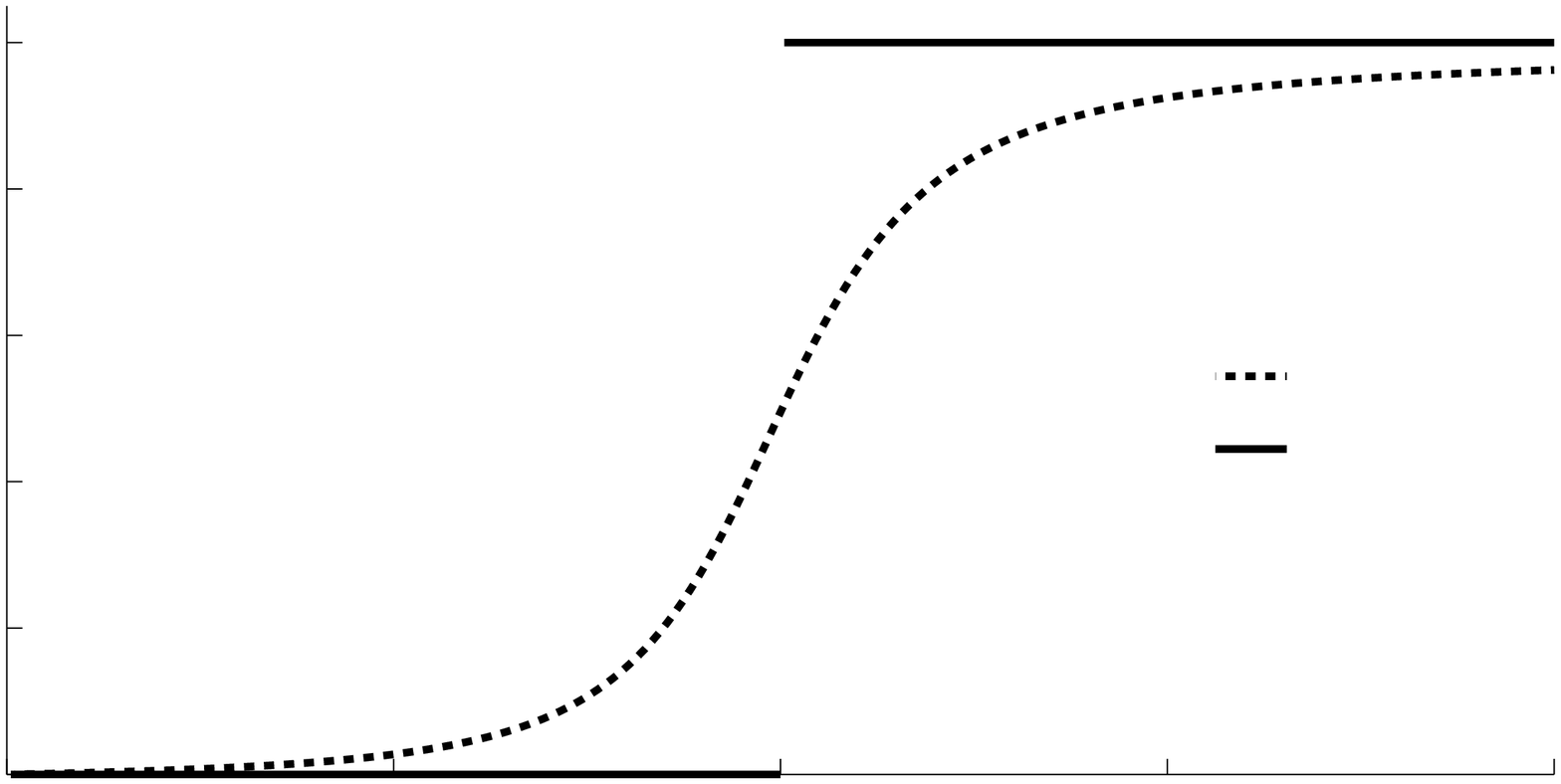}
\end{picture}%
\begin{picture}(800,401)(0,0)
\fontsize{20}{0}
\selectfont\put(104,58.0014){\makebox(0,0)[t]{\textcolor[rgb]{0,0,0}{{0}}}}
\fontsize{20}{0}
\selectfont\put(259,58.0014){\makebox(0,0)[t]{\textcolor[rgb]{0,0,0}{{0.5}}}}
\fontsize{20}{0}
\selectfont\put(414,58.0014){\makebox(0,0)[t]{\textcolor[rgb]{0,0,0}{{1}}}}
\fontsize{20}{0}
\selectfont\put(569,58.0014){\makebox(0,0)[t]{\textcolor[rgb]{0,0,0}{{1.5}}}}
\fontsize{20}{0}
\selectfont\put(724,58.0014){\makebox(0,0)[t]{\textcolor[rgb]{0,0,0}{{2}}}}
\fontsize{20}{0}
\selectfont\put(99,63.0014){\makebox(0,0)[r]{\textcolor[rgb]{0,0,0}{{0}}}}
\fontsize{20}{0}
\selectfont\put(99,121.668){\makebox(0,0)[r]{\textcolor[rgb]{0,0,0}{{0.2}}}}
\fontsize{20}{0}
\selectfont\put(99,180.334){\makebox(0,0)[r]{\textcolor[rgb]{0,0,0}{{0.4}}}}
\fontsize{20}{0}
\selectfont\put(99,239.001){\makebox(0,0)[r]{\textcolor[rgb]{0,0,0}{{0.6}}}}
\fontsize{20}{0}
\selectfont\put(99,297.667){\makebox(0,0)[r]{\textcolor[rgb]{0,0,0}{{0.8}}}}
\fontsize{20}{0}
\selectfont\put(99,356.333){\makebox(0,0)[r]{\textcolor[rgb]{0,0,0}{{1}}}}
\fontsize{30}{0}
\selectfont\put(414,36.0014){\makebox(0,0)[t]{\textcolor[rgb]{0,0,0}{{$u/s$}}}}
\fontsize{30}{0}
\selectfont\put(67,217.001){\rotatebox{90}{\makebox(0,0)[b]{\textcolor[rgb]{0,0,0}{{$g_{\infty}^{}(\bar{y})$}}}}}
\fontsize{30}{0}
\selectfont\put(620.671,226.568){\makebox(0,0)[l]{\textcolor[rgb]{0,0,0}{{$\ \ \nu_0=\frac{1}{100} $}}}}
\fontsize{30}{0}
\selectfont\put(620.671,191.432){\makebox(0,0)[l]{\textcolor[rgb]{0,0,0}{{$\ \ \nu_0=0$}}}}
\fontsize{20}{0}
\selectfont\put(410,357){\makebox(0,0)[l]{\textcolor[rgb]{0,0,0}{{$\bullet$}}}}
\end{picture}

%% file: Detlimit_Revised.bbl
\begin{thebibliography}{35}
	\providecommand{\natexlab}[1]{#1}
	\providecommand{\url}[1]{\texttt{#1}}
	\expandafter\ifx\csname urlstyle\endcsname\relax
	\providecommand{\doi}[1]{doi: #1}\else
	\providecommand{\doi}{doi: \begingroup \urlstyle{rm}\Url}\fi
	
	\bibitem[Athreya and Ney(1972)]{athreyaney1972}
	K.~B. Athreya and P.~Ney.
	\newblock \emph{Branching Processes}.
	\newblock Springer, Berlin, 1972.
	
	\bibitem[Athreya and Swart(2005)]{Athreya2005}
	S.~R. Athreya and J.~M. Swart.
	\newblock Branching-coalescing particle systems.
	\newblock \emph{Probab. Theor. Relat. Fields}, 131:\penalty0 376--414, 2005.
	
	\bibitem[Baake and Wiehe(1997)]{Baake1997}
	E.~Baake and T.~Wiehe.
	\newblock Bifurcations in haploid and diploid sequence space models.
	\newblock \emph{J. Math. Biol.}, 35:\penalty0 321--343, 1997.
	
	\bibitem[Baake et~al.(2016)Baake, Lenz, and Wakolbinger]{baake2016}
	E.~Baake, U.~Lenz, and A.~Wakolbinger.
	\newblock The common ancestor type distribution of a
	{$\Lambda$}-{W}right-{F}isher process with selection and mutation.
	\newblock \emph{Electron. Commun. Probab.}, 21:\penalty0 1--14, 2016.
	
	\bibitem[B{\"u}rger(2000)]{burger2000mathematical}
	R.~B{\"u}rger.
	\newblock \emph{The Mathematical Theory of Selection, Recombination, and		Mutation}.
	\newblock Wiley Chichester, 2000.
	
	\bibitem[Clifford and Sudbury(1985)]{clifford1985}
	P.~Clifford and A.~Sudbury.
	\newblock A sample path proof of the duality for stochastically monotone markov processes.
	\newblock \emph{Ann. Probab.}, 13:\penalty0 558--565, 1985.
	
	\bibitem[{Cordero}(2017)]{2015CorderoULT}
	F.~{Cordero}.
	\newblock {The deterministic limit of the Moran model: a uniform central limit
		theorem}.
	\newblock \emph{Markov Proc. Relat. Fields}, 23:\penalty0 313--324, 2017.
	
	\bibitem[Cordero(2017)]{Cordero2017590}
	F.~Cordero.
	\newblock Common ancestor type distribution: A {M}oran model and its
	deterministic limit.
	\newblock \emph{Stoch. Proc. Appl.}, 127:\penalty0 590--621, 2017.
	
	\bibitem[Crow and Kimura(1956)]{crow1956}
	J.~F. Crow and M.~Kimura.
	\newblock Some genetic problems in natural populations.
	\newblock \emph{Proc. Third Berkeley Symp. on Math. Statist. and Prob.},	4:\penalty0 1--22, 1956.
	
	\bibitem[Crow and Kimura(1970)]{crow1970introduction}
	J.~F. Crow and M.~Kimura.
	\newblock \emph{An Introduction to Population Genetics Theory.}
	\newblock Harper \& Row, New York, 1970.
	
	\bibitem[Donnelly and Kurtz(1999)]{donnelly1999}
	P.~Donnelly and T.G.~Kurtz.
	\newblock Particle representations for measure-valued population models.	\emph{Ann. Probab.}, 27:\penalty0 166--205, 1999.
	
	\bibitem[Durrett(2008)]{durrett2008probability}
	R.~Durrett.
	\newblock \emph{Probability Models for DNA Sequence Evolution, 2nd ed.}
	\newblock Springer, New York, 2008.
	
	\bibitem[Eigen(1971)]{Eigen1971}
	M.~Eigen.
	\newblock Self-organization of matter and the evolution of biological
	macromolecules.
	\newblock \emph{Naturwissenschaften}, 58:\penalty0 465--523, 1971.
	
	\bibitem[Eigen et~al.(1989)Eigen, McCaskill, and Schuster]{Eigenetal88}
	M.~Eigen, J.~McCaskill, and P.~Schuster.
	\newblock The molecular quasi-species.
	\newblock \emph{Adv. Chem. Phys.}, 75:\penalty0 149--263, 1989.
	
	\bibitem[Ethier and Kurtz(1986)]{ethier1986markovprocces}
	S.~N. Ethier and T.~G. Kurtz.
	\newblock \emph{Markov {P}rocesses: {C}haracterization and {C}onvergence}.
	\newblock J. Wiley \& Sons, 1986.
	
	\bibitem[Ewens(2004)]{ewens2004mathematical}
	W.~Ewens.
	\newblock \emph{Mathematical Population Genetics 1: Theoretical Introduction}.
	\newblock Springer, New York, 2004.
	
	\bibitem[Fearnhead(2002)]{fearnhead2002common}
	P.~Fearnhead.
	\newblock The common ancestor at a nonneutral locus.
	\newblock \emph{J. Appl. Probab.}, 39:\penalty0 38--54, 2002.
	
	\bibitem[Feller(1951)]{feller1951diffusion}
	W.~Feller.
	\newblock Diffusion processes in genetics.
	\newblock In J. Neyman (ed.): \emph{Proc. Second Berkeley Symp. Math. Stat. Prob. 1950.} Univ.
	California Press, Berkeley, Los Angeles, CA, 1951, pp. 227--246.
	
	\bibitem[Fisher(1930)]{fisher1930genetical}
	R.~A. Fisher.
	\newblock \emph{The Genetical Theory of Natural Selection}.
	\newblock Clarendon Press, Oxford, 1930.
	
	\bibitem[Georgii and Baake(2003)]{georgii2003}
	H.-O. Georgii and E.~Baake.
	\newblock Supercritical multitype branching processes: {T}he ancestral types of
	typical individuals.
	\newblock \emph{Adv. in Appl. Probab.}, 35:\penalty0 1090--1110, 2003.
	
	\bibitem[Jagers(1989)]{JAGERS1989183}
	P.~Jagers.
	\newblock General branching processes as {M}arkov fields.
	\newblock \emph{Stoch. Proc. Appl.}, 32:\penalty0 183--212, 1989.
	
	\bibitem[Jagers(1992)]{Jagers1992}
	P.~Jagers.
	\newblock Stabilities and instabilities in population dynamics.
	\newblock \emph{J. App. Prob.}, 29:\penalty0 770--780, 1992.
	
	\bibitem[Jansen and Kurt(2014)]{jansen2014}
	S.~Jansen and N.~Kurt.
	\newblock On the notion(s) of duality for {M}arkov processes.
	\newblock \emph{Probab. Surveys}, 11:\penalty0 59--120, 2014.
	
	\bibitem[Karlin and McGregor(1957)]{KarlinMcGregor57}
	S.~Karlin and J.~McGregor.
	\newblock The classification of birth and death processes.
	\newblock \emph{Trans. Amer. Math. Soc.}, 86:\penalty0 366--400, 1957.
	
	\bibitem[Krone and Neuhauser(1997)]{krone1997ancestral}
	S.~M. Krone and C.~Neuhauser.
	\newblock Ancestral processes with selection.
	\newblock \emph{Theor. Popul. Biol.}, 51\penalty0:\penalty0 210--237, 1997.
	
	\bibitem[Kurtz(1970)]{kurtz1970}
	T.~G. Kurtz.
	\newblock Solutions of ordinary differential equations as limits of pure jump
	{M}arkov {P}rocesses.
	\newblock \emph{J. Appl. Probab.}, 7:\penalty0 49--58, 1970.
	
	\bibitem[Lenz et~al.(2015)Lenz, Kluth, Baake, and Wakolbinger]{Lenz201527}
	U.~Lenz, S.~Kluth, E.~Baake, and A.~Wakolbinger.
	\newblock Looking down in the ancestral selection graph: A probabilistic	approach to the common ancestor type distribution.
	\newblock \emph{Theor. Popul. Biol.}, 103:\penalty0 27--37, 2015.
	
	\bibitem[Liggett(2010)]{liggett2010}
	T.~M. Liggett.
	\newblock \emph{Continuous Time Markov Processes. An Introduction}.
	\newblock American Mathematical Society, RI, 2010.
	
	\bibitem[Mal{\'e}cot(1948)]{malecotmathematiques}
	G.~Mal{\'e}cot.
	\newblock \emph{Les {M}ath{\'e}matiques de l'{H}{\'e}r{\'e}dit{\'e}}.
	\newblock Masson, Paris, 1948.
	
	\bibitem[Mano(2009)]{Mano2009164}
	S.~Mano.
	\newblock Duality, ancestral and diffusion processes in models with selection.
	\newblock \emph{Theor. Popul. Biol.}, 75:\penalty0 164--175, 2009.
	
	\bibitem[Moran(1958)]{moran1958random}
	P.~A.~P. Moran.
	\newblock Random processes in genetics.
	\newblock \emph{Math. Proc. Cam. Phil. Soc.}, 54:\penalty0 60--71, 1958.	
	\bibitem[Norris(1998)]{norris1998markov}
	J.~R. Norris.
	\newblock \emph{Markov Chains, $2$nd ed.}
	\newblock Cambridge Univ. Press, Cambridge, 1998.
	
	\bibitem[Shiga and Uchiyama(1986)]{Shiga1986}
	T.~Shiga and K.~Uchiyama.
	\newblock Stationary states and their stability of the stepping stone model
	involving mutation and selection.
	\newblock \emph{Probab. Theor. Relat. Fields}, 73:\penalty0
	87--117, 1986.
	
	\bibitem[Taylor(2007)]{Taylor2007}
	J.~E. Taylor.
	\newblock The common ancestor process for a {W}right-{F}isher diffusion.	\newblock \emph{Electron. J. Probab.}, 12:\penalty0 808--847, 2007.
	
	\bibitem[Wakeley(2009)]{wake2009}
	J.~Wakeley.
	\newblock \emph{Coalescent Theory: An Introduction}.
	\newblock Roberts \& Co, Greenwood Village, 2009.
	
	\bibitem[Wright(1931)]{wright1931evolution}
	S.~Wright.
	\newblock Evolution in {M}endelian populations.
	\newblock \emph{Genetics}, 16:\penalty0 97--159, 1931.
	
	\bibitem[Zanini et al.(2017)Zanini, Puller, Brodin, Albert, and Neher]{zanini2017vivo}
	F.~Zanini, V.~Puller, J.~Brodin, J.~Albert, and R.A.~Neher.
	\newblock \emph{In vivo} mutation rates and the landscape of fitness costs of HIV-$1$.
	\newblock \emph{Virus Evolution}, 3:\penalty0 vex003, 2017.
\end{thebibliography}
